\documentclass[12pt, a4paper, parskip=half, abstracton]{scrartcl}

\usepackage{array}
\usepackage{marginnote}
\usepackage{xcolor}
\usepackage{amscd,amssymb,amsfonts,amsmath,latexsym,amsthm}
\usepackage{hyperref}
\usepackage[all,cmtip]{xy}
\usepackage{mathrsfs}
\usepackage{graphicx}
\usepackage{bm} 
\usepackage[nameinlink, capitalise]{cleveref}

\usepackage{etoolbox}

\def\hB{\hspace*{\fill}$\qed$}

\usepackage[nottoc]{tocbibind} 

\usepackage{defs_pp1}
\usepackage{slashed}
\usepackage[utf8]{inputenc}
\usepackage{microtype}
\usepackage[english]{babel}
\usepackage{mathtools}
\usepackage{bm}
\usepackage{esvect}

\title{Coarse cone quotients}
\author{
Ulrich Bunke\thanks{Fakult{\"a}t f{\"u}r Mathematik,
Universit{\"a}t Regensburg,
93040 Regensburg,
ulrich.bunke@mathematik.uni-regensburg.de} 
}

\numberwithin{equation}{section}
\newtheorem{theorem}{Theorem}[section] 
\newtheorem{prop}[theorem]{Proposition}
\newtheorem{lem}[theorem]{Lemma}

\newtheorem{ddd}[theorem]{Definition}
\newtheorem{kor}[theorem]{Corollary}
\newtheorem{ass}[theorem]{Assumption}

\newtheorem{construction}[theorem]{Construction}

\theoremstyle{remark}
\theoremstyle{definition}

\newtheorem{ex}[theorem]{Example}
\newtheorem{rem}[theorem]{Remark}

\newcommand{\cass}{\mathrm{cass}}

\newcommand{\grp}{\mathrm{grp}}

\newcommand{\bgeom}{\mathrm{bgeom}}

\newcommand{\can}{\mathrm{can}}

\newcommand{\UBC}{\mathbf{UBC}}

\newcommand{\Yo}{\mathrm{Yo}}

\newcommand{\ho}{\mathrm{ho}}

\newcommand{\BC}{\mathbf{BC}}

\newcommand{\Fin}{\mathbf{Fin}}

\newcommand{\Cofib}{\mathrm{Cofib}}

\newcommand{\incl}{\mathrm{incl}}

\newcommand{\uli}[1]{\textcolor{red}{#1}}

\newcommand{\cZ}{{\mathcal{Z}}}

\newcommand{\cW}{{\mathcal{W}}}

\newcommand{\bA}{{\mathbf{A}}}

\newcommand{\cO}{{\mathcal{O}}}

\newcommand{\cY}{{\mathcal{Y}}}

\newcommand{\bP}{\mathbf{P}}

\newcommand{\strg}{\mathrm{strg}}

\newcommand{\cone}{\mathrm{cone}}

\newcommand{\cp}{\mathrm{cp}}

\renewcommand{\tr}{\mathrm{tr}}
\renewcommand{\id}{\mathrm{id}}

\newcommand{\CMtr}{\mathbf{CM}_{\mathrm{tr}}}
\newcommand{\CM}{\mathbf{CM} }
\newcommand{\CMctr}{\mathbf{CM}_{\mathrm{tr}}}

\newcommand{\BCtr}{\mathbf{BC}_{\mathrm{tr}}}

\newcommand{\str}{\mathrm{str}}

\newcommand{\disc}{\mathrm{disc}}

\newcommand{\Yoctr}{\mathrm{\Yo}_{\mathrm{tr}}}

\renewcommand{\Sq}{\mathrm{Sq}}
\newcommand{\geom}{\mathrm{geom}}

\begin{document}
 \setcounter{tocdepth}{1}

\maketitle

\begin{abstract}  We study the coarse motive of the quotient $\cO^{\infty}(X)//G$ of the cone of a uniform bornological coarse space $X$ with $G$-action. If $X$ admits a sufficiently ergodic probability measure, then we show  that the coarse assembly map for  $\cO^{\infty}(X)//G$ is not an equivalence. 
The main ideas are taken from a recent paper by C. Kitsios, T. Schick, and F. Vigolo \cite{Kitsios:2025aa} and are  adapted to the formalism of coarse homotopy theory based on bornological coarse spaces developed by A. Engel and the author \cite{buen}, \cite{ass}.

  \end{abstract}

 \tableofcontents

\section{Introduction}

Coarse geometry was introduced by J. Roe \cite{roe_index_coarse}, \cite{MR1147350}, \cite{roe_lectures_coarse_geometry} to study the large-scale  geometry of metric spaces, groups, and more general objects in a way which neglects  precise  numerical scales for the distances. Compatible bornologies as an additional structure were   proposed in \cite{buen}       in order  to fix notions of local finiteness.  Consequently, one considers 
 the category of bornological coarse spaces $\BC$ and proper controlled maps    as the general framework for developing homotopy-theoretic methods in coarse geometry. The natural framework to  study the small-scale properties of metric spaces globally without   fixing precise numerical scales for distances is built from uniform spaces. Compatible bornologies and  coarse structures are again added to capture local finiteness and some large-scale flavors. We propose to develop homotopy theory within  the category $\UBC$ of uniform bornological coarse spaces and uniform, proper, and controlled maps \cite{buen}, \cite{ass}.  
 
 These two categories are related by the forgetful functor 
\begin{equation}\label{fwerfvsdv}c:\UBC\to \BC
\end{equation}
which forgets the uniform structure.
More interestingly, we have the
geometric cone-at-infinity functor $$\cO^{\infty}:\UBC\to \BC$$
described in \cref{lpozhrtrger}. It
 can be used as a tool to translate local geometry into large-scale geometry.
This feature is best expressed  in terms of homology theories.

In \cite{buen}, coarse homology theories  were axiomatized as functors from  $\BC$ to cocomplete stable $\infty$-categories  that
satisfy coarse invariance, excision, vanishing on flasques, and $u$-continuity.  If a   coarse homology theory   annihilates  weakly flasque bornological coarse spaces,  it is called strong. 

 In order to state   assertions which hold for all coarse homology theories at once 
  we will work with the
 universal coarse homology theory and the universal strong coarse homology theory
\begin{equation}\label{gjiwejrgoewrijogwerf}\Yo:\BC\to \CM\ , \qquad \Yo^{\strg}:\BC\to \CM^{\strg}\ ,
\end{equation}
whose targets are called the categories of coarse motives and strong coarse motives \footnote{In ordinary topology, the universal homology theory is the suspension spectrum functor $\Sigma_{+}^{\infty}:\Top\to \Sp$. In   \cite{buen} we denoted coarse motives by   $\Sp\cX$ and called them coarse spectra in order to highlight the analogy with topology.}. These exist for formal reasons.

Similarly, local\footnote{This should remind of {\em locally  finite}. We use the word {\em local} since we replaced the local finiteness condition by the weaker condition of vanishing on flasques. The latter condition does not involve limits. It is therefore  better suited for motivic considerations.} homology theories are functors  from $\UBC$  to cocomplete stable $\infty$-categories
that are homotopy invariant, excisive,  $u$-continuous, and vanish on flasques \cite{ass}.
We again consider the  universal local homology theory 
\begin{equation}\label{werfeferfsfvfd}\Yo\cB:\UBC\to \Sp\cB\ .
\end{equation}
The role of the cone-at-infinity functor $\cO^{\infty}$ can now be clarified by the observation  \cite[Lem. 9.6]{ass} that the functor
\begin{equation}\label{sfdpokopsdvdsfvsdfvsdf}\cO^{\infty,\strg}:\UBC\xrightarrow{\cO^{\infty}} \BC\xrightarrow{\Yo^{\strg}} \CM^{\strg}
\end{equation}is a local homology theory. 

\begin{rem}
It is important to use the strong version $\Yo^{\strg}$ in \eqref{sfdpokopsdvdsfvsdfvsdf}.
For $X$ in $\UBC$ the shift  
$(t,x)\to (t+1,x)$ on $[0,\infty)\otimes X$ induces, by functoriality of $\cO^{\infty}$, an endomorphism $f$  
of $\cO^{\infty}([0,\infty)\otimes X)$ in $\BC$ that  is shifting and non-expanding. We have $\Yo(f)\simeq \id_{\cO^{\infty}([0,\infty)\otimes X)}$, but $f$    is not
close to the identity. Thus, $f$ only witnesses the weak flasqueness of  
$\cO^{\infty}([0,\infty)\otimes X)$. We can conclude that    $ \cO^{\infty,\strg}([0,\infty)\otimes X)\simeq 0$ as required by the axioms for a local homology theory.
 On the other hand we do not expect that   
  $\Yo( \cO^{\infty}([0,\infty)\otimes X))\simeq 0$ in general.  
  \hB \end{rem}


Let $G$ be a group. Colimits over $BG$ in $\UBC$ may fail to exist or can be  badly behaved.
 For example, the translation action of $\Z$ on $\R$  (with the metric  $\UBC$-structure)    admits no  quotient  $\colim_{B\Z}\R$  because the $\Z$-orbits are   unbounded. To see this, assume for a  contradiction that a quotient  $\colim_{B\Z}\R$ exists. 
For any $X$ in $\UBC$, a morphism
$\colim_{B\Z}\R\to X$ would then correspond to 
a $\Z$-equivariant morphism $f:\R\to X$, where $X$ carries the trivial $\Z$-action. Every $\Z$-orbit in $\R$ is sent to a single point. Since $f$ is proper, we would conclude that every $\Z$-orbit in $\R$  is bounded, which is absurd. Thus,
$\Hom_{\UBC}(\colim_{B\Z}\R,X)\cong \emptyset$  for every $X$ in $\UBC$. Setting
  $X=\colim_{B\Z}\R$, we see that this is impossible.

If we consider  $S^{1}$ (again with the metric $\UBC$-structure) and an irrational rotation action of $\Z$, then
$\colim_{B\Z}S^{1}$ is the quotient set $S^{1}/\Z$ with the uninteresting    maximal coarse and bornological
 structures and the minimal uniform structure. 

In \cref{kokptoerhertgretgrtgrtg}, we describe a functor
  $$-//G:\Fun(BG,\BC)\to \BC$$  and show that it
 represents the homotopy orbits in the homotopy theory on $\BC$ generated by the coarse equivalences; see \cref{kopgwergwerfw} for a precise formulation.
It is then natural to ask whether the  composition 
$$\cO^{\infty}(-)//G:\Fun(BG,\UBC)\to \BC$$ is a good replacement of the problematic
$\cO^{\infty}( \colim_{BG}-)$.
 
 \begin{rem} 
 Recall that, for $X$ in $\UBC$, the cone-at-infinity $\cO^{\infty}(X)$ is the set $\Z\times X$ equipped with a  bornological coarse structure derived from the uniform bornological coarse structure of $X$  (see 
 \cref{lpozhrtrger}), and that $\cO^{\infty}_{\ge 0}(X)$ is the subset $\nat\times X$ endowed with the induced bornological coarse structure. 
 
 The subspace $\cO^{\infty}_{\ge 0}(X)//G\subset\cO^{\infty}(X)//G$ is a version (see \cref{hezkohprherthgeg} for the relation between the universal cones and cones with a fixed scale as usually considered in the literature) of the warped cone introduced for metric spaces with $G$-action by J. Roe. The coarse geometry of these warped cones has been 
studied in a variety of papers, see, e.g., \cite{dsa}, \cite{Vigolo_2018},   \cite{de_Laat_2018}, \cite{Fisher_2019},   \cite{Li_2023}, \cite{Kitsios:2025aa}.  
 In the metric  context,  the Lipschitz geometry of the warped cones  
  is a very fine invariant, as for instance demonstrated in \cite{Fisher_2019}, \cite{Sawicki_2020}.  We think that a similar theory could be developed also for $\cO_{\ge 0}^{\infty}(X)//G$ as an object in $\BC$. But in this note,   we focus on  properties 
  that  can be  detected by coarse homology theories.
We furthermore prefer to work with $\cO^{\infty}(X)//G$ instead of 
$\cO^{\infty}_{\ge 0}(X)//G$ since the former  is a better invariant of the  local geometry of the $G$-object  $X$ alone. In contrast to $\Yo(\cO^{\infty}_{\ge 0}(X)//G)$ the motive $\Yo(\cO^{\infty}(X)//G)$  is,   e.g.,  invariant under enlarging the coarse structure of $X$.  
\hB\end{rem}

In this paper, we use the Rips complex functor $\bP$, the cone-at-infinity functor, and the cone boundary to define the 
motivic coarse assembly map
\begin{equation}\label{gergwerfrefwfw}\mu: \Sigma^{-1} \cO^{\infty,\str} \bP\to \cp:\CM\to \CM^{\strg}
\end{equation}
as proposed in \cite{ass}, see \cref{okopgtrgegrtg}. 
Let $$\CM_{\disc}\subseteq \CM_{\cass}\subseteq \CM$$ be the localizing subcategories  generated by the motives of discrete bornological coarse spaces, or consisting of the  motives for which the motivic coarse assembly map $\mu$  is an equivalence; see  \cref{kohperhretgegtre}. Interpreting 
 \cite{Kitsios:2025aa} in the context described above, we provide examples of $X$ in $\Fun(BG,\UBC)$ such that
 $\Yo(\cO^{\infty}(X)//G)$ does not belong to $ \CM_{\cass}$.

 Let $X$ be in $\Fun(BG,\UBC)$.
 \begin{theorem}\label{kopherthretge9}
 Assume:
 \begin{enumerate}
  
  \item\label{lptgeertgetrg} $X$ is bornologically bounded and has the maximal coarse structure.   
  \item \label{hpkprhtrgrtegerg1} Every uniform entourage $V$ of $X$ admits a finite $V$-dense subset.
\item \label{hpkprhtrgrtegerg} $X$ admits a uniform scale (see \cref{okbepgrbkeprbegfdb}).
 \item \label{rkhoeprtgertgretgegtr} Every uniform scale of $X$ is dominated by a Lipschitz scale (for the $G$-action)
 of finite Assouad-Nagata dimension   (see  \cref{kohpertgetrgeg} and \cref{asna}).  
 \item \label{ojoeprtherthe} The $G$-action on $X$ is uniformly free (see \cref{jkrtpzrtzhjrthrtzhrth}).
 \item \label{fhqwekfqweq} $G$ is finitely generated and has finite asymptotic dimension.

 \item \label{kpbgbrgebgrbgrb}$X$ admits an invariant non-atomic Borel probability measure $\nu$ with $\supp(\nu)=X$ such that $G$ acts ergodically on $(X,\nu)$ and  
 the unitary action on  $L^{2}(X,\nu)$ has a spectral gap (see \cref{gkwopergwerferfw}).
    \end{enumerate}
Then $\Yo(\cO^{\infty}(X)//G)$ does not belong to $ \CM_{\cass}$. 
   \end{theorem}

   Note that the Assumptions \ref{kopherthretge9}.\ref{lptgeertgetrg}-\ref{hpkprhtrgrtegerg} are satisfied if $X$  is represented by a finite union of  compact  path-metric   spaces. In this case, the
    Assumptions \ref{kopherthretge9}.\ref{rkhoeprtgertgretgegtr}-\ref{ojoeprtherthe} hold if, in addition, $G$ acts  on $X$ freely by Lipschitz maps, and $X$ has finite Assouad-Nagata dimension.

 Assumption  \ref{kopherthretge9}.\ref{kpbgbrgebgrbgrb} is used to show that the Drutu-Nowak projection 
$\hat P$ from \cref{fewfewrdefwerfwef} belongs to the Roe algebra.  It
could be weakened to the assumption that the action of $G$ on $(X,\nu)$ is strongly ergodic; see \cref{erwg9upwerg}.

  \begin{ex}
  Let $G$ be a non-abelian free subgroup of $SU(2,\bar \Q)$  acting on $X:=SU(2)$.
  The space $X$ is a compact Riemannian manifold, and therefore it is a compact path-metric space.
 The group $G$ is finitely generated, and it acts freely and isometrically on $X$. 
 Since it is free, it has finite asymptotic dimension.
The action has a spectral gap by  \cite{zbMATH05236752}.
Therefore, the action of $G$ on $X$ satisfies the assumptions of \cref{kopherthretge9}.
 \hB

  \end{ex}

  \begin{rem}\label{hezkohprherthgeg}
In the present paper, we prefer to work with the universal cone-at-infinity $\cO^{\infty}$ 
where we allow an arbitrarily slow decay (measured with the uniform structure) of the entourages to the right,
since it can be defined as a functor from $\UBC$ to $\BC$.
If the uniform structure  of $X$ admits a countable cofinal family as required by Assumption  \ref{kopherthretge9}.\ref{hpkprhtrgrtegerg}, then one could fix the decay rate $\phi$ (this is the same as a uniform scale \cref{okbepgrbkeprbegfdb})  and consider the version
$  \cO_{\ge 0,\phi}(X)$ of the positive part of the cone  as in \cite[Def. 8.5]{ass} (denoted there by $\tilde \cO_{\phi}(X)$). The euclidean cones   for metric spaces considered in the literature are examples. 
We let $\cO_{\phi}^{\infty}(X)$ be the extension of $\cO_{\ge 0,\phi}(X)$ to the left by a cylinder.
Fixing the decay destroys the functoriality  on $\UBC$. It could also destroy  the $\Z$-action  on $\cO^{\infty}_{\phi}(X)$ by shifts, but this does not happen   for    the euclidean  decay.
By \cite[Lem. 8.7]{ass}, fixing the decay rate does not change the coarse motive of the cone. 
The same rescaling argument as in the proof  of  \cite[Lem. 8.7]{ass}   also shows that  $$\Yo(\cO^{\infty}_{\phi}(X)//G)\simeq \Yo(\cO^{\infty}(X)//G)\ .$$
Therefore,  \cref{kopherthretge9} also shows that
$\Yo(\cO^{\infty}_{\phi}(X)//G)$ is not in $\CM_{\cass}$. 
This, e.g., applies to the euclidean cone for metric spaces.
\hB
\end{rem}

 The definition of $\cO^{\infty}(X)$ and many arguments in the present note employ the   squeezing space
 $\Sq(X)$ introduced in \cref{ugweogwregw}.
  The  technical-looking Assumptions \ref{kopherthretge9}.\ref{lptgeertgetrg}-\ref{fhqwekfqweq} are used to ensure that
    $\Sq(X)//G$  has a cofinal set of  coarse entourages $V$ such that
   $(\Sq(X)//G)_{V}$  has bounded geometry and   its coarse $G$-covering $((G\ltimes \Sq(X))//G)_{U}$
   from \cref{uoiwgrgwegwerfwfwerf}
   has finite asymptotic dimension (here $U$ is derived from $V$ and the subscripts like $(-)_{U}$ indicate that the set is equipped with the coarse structure generated by $U$).  
   These conditions also ensure that
 $(\cO^{\infty}(X)//G)_{V}$ has bounded geometry,  and that  $((G\ltimes \cO^{\infty}(X))//G)_{U}$ has finite asymptotic dimension 
 for a cofinal set of coarse entourages $V$ of $\cO^{\infty}(X)//G$.
   Since finite asymptotic dimension implies that the coarse assembly map is an equivalence,
 \cref{kopherthretge9} (respectively, its proof) shows that $ \cO^{\infty}(X)//G $ (or $\Sq(X)//G$) 
  does not admit a  cofinal set of coarse entourages $V$ such that  $ (\cO^{\infty}(X)//G)_{V}$ (or $(\Sq(X)//G)_{V}$)
  has 
   finite asymptotic dimension, i.e., they do not have  weakly finite asymptotic dimension in the sense of \cite[Def. 10.3]{ass}.

  In order to state the assumptions in a less technical manner in  \cite{Kitsios:2025aa}, 
 it was simply assumed that $X$ is presented by a compact connected Riemannian manifold  on which $G$ acts freely  by Lipschitz maps. Furthermore,  Assumption  \ref{kopherthretge9}.\ref{fhqwekfqweq}
 is weakened to the assumption that $G$ has Property $A$. This is possible since in  \cite{Kitsios:2025aa}, instead of finite asymptotic dimension, 
 the operator norm localization property is used to define the transfer. In this note, however,  we stick to finite asymptotic dimension in order to be able to cite \cite{Bunke:2025aa}.

  In order to show that a motive $Y$  in $\CM$ does not belong to $\CM_{\cass}$ it suffices to consider any spectrum-valued strong coarse homology theory  $E:\BC\to \Sp$, or equivalently, the corresponding colimit-preserving functor $E:\CM^{\strg}\to \Sp$, and to show that $$E(\mu_{Y}): \pi_{*+1}E(\cO^{\infty,\strg}\bP(Y))\to \pi_{*}E  (Y)$$ is not an isomorphism of groups.
 For the proof of  \cref{kopherthretge9}, we  employ the coarse $K$-homology theory $K\cX:\BC\to \Sp$ 
  as introduced in \cite{buen}, \cite{coarsek}.
The assumptions of  \cref{kopherthretge9}, in particular   \ref{kopherthretge9}.\ref{kpbgbrgebgrbgrb},  are  used to produce a  class $u$ in $\pi_{1}K\cX(\cO^{\infty}(X)//G)$   
 which does not belong to the image of 
 coarse assembly map $$K\cX(\mu_{\cO^{\infty}(X)//G}):\pi_{2} K\cX(\cO^{\infty,\strg}\bP(\cO^{\infty}(X)//G)) \to \pi_{1}K\cX(\cO^{\infty}(X)//G)\ .$$ 
 The proof uses transfers along branched coarse  coverings and traces in the version of  \cite{Bunke:2025aa}.  The construction of the class $u$  and the subsequent arguments
 are an  interpretation of \cite[(4.2)]{Kitsios:2025aa} in the current context, and will be carried out in  \cref{hkoeprttrgegertg}.

Note that  \cref{kopherthretge9} cannot be deduced directly from  \cite{Kitsios:2025aa} in a technical sense.
First of all, the  classical coarse $K$-homology functor used in the reference is not defined on all of $\BC$ and also not spectrum-valued. Furthermore, the version of the coarse assembly map
used in  \cite{Kitsios:2025aa} is not a natural transformation between coarse homology theories
in the sense of \cite{buen}. Thus the fact that the coarse assembly map in   the examples of warped cones considered in \cite{Kitsios:2025aa} is not an isomorphism  cannot be used to conclude that
the motives of these warped cones do not belong to $\CM_{\cass}$. The main purpose of the present note
is to demonstrate that, nevertheless, the ideas from \cite{Kitsios:2025aa}  can be adapted to work in the motivic context.

 The \cref{kopherthretge9} supports  the expectation that
 the complexity of the $G$-action on $X$ is non-trivially reflected in the complexity of the coarse  motive  
  $\Yo(\cO^{\infty}(X)//G)$.     
  In view of the variety of examples of $X$ discussed, e.g., in \cite{Fisher_2019}, 
 its is natural  to ask whether they give rise to different motives 
 $\Yo(\cO^{\infty}(X)//G)$ in $\CM$.  We hope to discuss this problem in the future.
 
 The proof of \cref{kopherthretge9}  is based on an application of the $L^{2}$-index theorem
 to sequence spaces. At this point, the  
  current literature refers to \cite[Lem. 6.5]{Willett_2012}, which works  in the context of the classical version of the coarse assembly map whose domain is expressed in terms of analytic $K$-homology. In  \cref{retgtrtbvgfdbgfd}, we provide the analogue of this result for  the coarse assembly map $K\cX(\mu)$ with $\mu$ as in \eqref{gergwerfrefwfw}, which is based on the version of the $L^{2}$-index theorem shown in  \cite{Bunke:2025aa}; see also \cref{hkoperhtetrgetgrt}. Since the complete statement is
  quite technical, we refrain from reproducing it here in the introduction.

 We now describe the contents of the paper in greater detail.

 In \cref{kohpehrtgergtrg}, we introduce the functor $-//G$ and characterize it by universal properties.

In \cref{koppgwerefwerfwerf}, we introduce the  squeezing space functor $$\Sq:\UBC\to \Fun(B\Z,\BC)\ .$$ 
The cone-at-infinity functor is then derived from the squeezing space functor
by
$$\cO^{\infty}(-):=\Sq(-)//\Z\ .$$

In \cref{gojpergewrfwefwf} we calculate the motive of the cone $\cO^{\infty}(X)//G$ in terms of the motive of the quotient of the squeezing space $\Sq(X)//G$.   The squeezing space  comes with a natural big family
$\Sq_{-}(X)//G$ and the $\Z$-action on $\Sq(X)//G$ induces one on the motive
$\Yo(\Sq(X)//G,\Sq_{-}(X)//G)$ in $\CM$. The motivic interpretation of \cite[Sec. (3.2)]{Kitsios:2025aa}
is stated in \cref{lkprhrhertgertgertge} and asserts  an equivalence $$ \Yo(\cO^{\infty}(X)//G)\stackrel{\simeq}{\to}  \colim_{B\Z} \Yo(\Sq(X)//G,\Sq_{-}(X)//G)^{\sign}\ .$$ in $\CM$.
 
 In \cref{joopwbergregrfwrf}, we show under the assumption that the $G$-action on $X$ is uniformly  free and further conditions,
 that $(\Sq(X)//G,\Sq_{-}(X)//G)$ and $(\cO^{\infty}(X)//G,\cO_{-}^{\infty}(X)//G)$ support branched coarse $G$-coverings 
 in the sense of \cite{Bunke:2025aa}. 
The transfer in coarse $K$-homology along these coverings is a tool going into the proof of    \cref{kopherthretge9}.
 
 In \cref{kopggwregweg}, we  consider the universal coarse homology with transfers 
\begin{equation}\label{vsdfcsdvdfvsdfver}\ \Yoctr:\BCtr \to \CMctr\ ,
\end{equation}  see \cite{trans}. 
 We construct actions of the rings
 $$\prod_{\Z} \Z\rtimes \Z\ , \qquad \frac{\prod_{\Z} \Z}{\bigoplus_{\Z} \Z} \rtimes \Z $$ on the motives $ \Yoctr(\Sq(X)//G)$ or 
 $\Yoctr(\Sq(X)//G,\Sq_{-}(X)//G)$ in $\ho(\CMctr)$, respectively. 
We can use these actions in order to manipulate classes in $E(\Sq(X)//G,\Sq_{-}(X)//G)$ for any coarse homology theory with transfers $E$. This, in particular, applies to  the coarse $K$-homology $K\cX$ since it  extends to a coarse homology with transfers   \cite{coarsek}.

In \cref{kohperhretgegtre}, we recall from \cite{ass}  the construction  of the motivic 
 coarse assembly map \eqref{gergwerfrefwfw}.

In \cref{hkoeprttrgegertg1}, we state and prove the $L^{2}$-index theorem for sequence spaces \cref{retgtrtbvgfdbgfd}.

In \cref{hkoeprttrgegertg}, we show  \cref{kopherthretge9} modulo the construction of suitable $K$-theory classes on the squeezing space.

  Finally, in \cref{kophjkertophrtgertgrteg9}, we use the spectral gap  of the action of $G$ on $L^{2}(X,\nu)$
in order to construct the  coarse $K$-homology class $p$
in $\lim_{B\Z} \pi_{0}K\cX(\Sq(X)//G)$ used to show \cref{kopherthretge9}.
Versions of this class have been considered before in   \cite{drno}, \cite{dsa}, \cite{Li_2021}, \cite{Li_2023}, \cite{Kitsios:2025aa} and our contribution here is to put these constructions into the  context of the coarse $K$-homology functor $K\cX$.

{\em Acknowledgement:  The author was supported by the SFB 1085 (Higher Invariants) funded by the Deutsche Forschungsgemeinschaft (DFG).  He  thanks Th. Schick and F. Vigolo for interesting discussions on \cite{Kitsios:2025aa}, which motivated him to write these notes. }

\section{Quotients}\label{kohpehrtgergtrg}

We consider the  symmetric monoidal category of bornological coarse spaces $\BC$ as introduced in \cite{buen}. 
Declaring close maps to be equal, we define a  quotient category
$$q:\BC\to \BC_{h}\ .$$
The following proposition clarifies the universal property of this construction.  
\begin{prop}[\cite{Heiss:2019aa}] The  quotient
functor $q:\BC\to \BC_{h}$ is the Dwyer-Kan localization at the coarse equivalences.
\end{prop}

In the following, we provide models for the homotopy coequalizer and homotopy quotients by group actions for the homotopy theory on $\BC$ generated by the coarse equivalences.
We start with coequalizers.
We consider two maps  \begin{equation}\label{ojopwgefwr}
\xymatrix{Y\ar@/^0.3cm/[r]^{f}\ar@/_0.3cm/[r]_{g}&X}
\end{equation}  
in $\BC$. \begin{ddd}
We let $X//(f,g)$ denote the set $X$ with the bornology induced from  the bornology of $X$ and the coarse structure generated by the coarse structure of $X$ and the entourage $\{(f(y),g(y))|y\in Y\}$.\end{ddd}

If $f$ and $g$  are  bornological maps, i.e., maps sending bounded sets to bounded sets (note that this condition only depends on their equivalence classes in $\BC_{h}$), then this  possibly bigger coarse structure  is still
compatible with the bornology, and  $X//(f,g)$ is again a bornological coarse 
space. In this case, the identity map of the underlying sets   is a morphism $e:X\to X//(f,g)$  in $\BC$.

\begin{lem}\label{kopheergtrge} If $f$ and $g$   are bornological maps, then \begin{equation}\label{ojopwgefwr1}
\xymatrix{Y\ar@/^0.3cm/[r]^{q(f)}\ar@/_0.3cm/[r]_{q(g)}&X\ar[r]^-{q(e)}&X//(f,g)}
\end{equation}  
 is a coequalizer diagram in $\BC_{h}$.
\end{lem}
\begin{proof}
One checks in a straightforward manner that $q(e)$ has the desired universal property. See the proof of \cref{kopgwergwerfw} for a slightly more detailed analogous argument.
\end{proof}

Let $G$ be a group, and let $BG$ denote the category with one object having $G$ as its endomorphisms.     \begin{ddd} \label{kokptoerhertgretgrtgrtg} We define  the  functor
$$-//G:\Fun(BG,\BC)\to \BC$$  such that the 
 underlying bornological space of $X//G$ is that of $X$, and 
 the coarse structure  is generated by the coarse structure of $X$ and
the entourages $U_{g}:=\{(gx,x)\mid x\in X\}$ of $X$ for all $g$ in $G$. 
 \end{ddd}
 
  One  checks that the bornology on $X//G$ is compatible with the coarse structure, so that the functor is well-defined.
Note that every $g$ in $G$ acts on $X$ by an automorphism and is therefore bornological.
 
\begin{lem}\label{kopgwergwerfw}
We have a canonical equivalence
$$q(X//G)\simeq \colim_{BG} q(X)\ .$$
 \end{lem}
\begin{proof}
Let $Y$ be in $\BC$ and $f:X\to Y$ be a morphism in $\BC$ such that $q(f)$ is $G$-invariant.
We consider the diagram
$$ 
\xymatrix{q(X)\ar[dr]\ar[rr]^{q(f)}&&q(Y)\\&q(X//G)\ar@{..>}[ur]&}\ .$$
One first checks that the underlying  setmap of $f$ can be used to produce the dotted arrow.
Furthermore, any other choice  of such a map is close to $f$ and therefore equal  to $f$ in $\BC_{h}$.
 \end{proof}

\begin{ex}\label{gjiwioergwrf}
A   coarsely invariant   functor $E:\BC\to \cC$ has a unique factorization
$$ 
\xymatrix{\BC\ar[dr]_{q}\ar[rr]^{E}&&\cC\\&\BC_{h}\ar@{..>}[ur]&}\ .$$
Assume that $\cC$ admits colimits indexed by $BG$. For  $X$ in $\Fun(BG,\BC)$ we then get an assembly morphism
\begin{equation}\label{gerfwerfwrwre}\xymatrix{&E(X)\ar[dr]\ar[dl]_{\can}&\\\colim_{BG}E(X)\ar@{..>}[rr]&&E(X//G)} 
\end{equation}
indicated by the dotted arrow, where $\can$ is the canonical map to the colimit, and the down-right arrow is induced by the morphism $X\to X//G$.  \hB \end{ex}

If $H$ is a second group, then the results above extend to the category $H\BC$ (see \cite{equicoarse})
of  
 $H$-bornological coarse spaces  with a $G$-action by automorphisms.
 
 \begin{ex} \label{hokeprthgrtgertgt} 
 For a set $X$, we let $X_{\min,\min}$ denote the bornological coarse space with the minimal coarse and bornological structures. The coarse entourages of $X_{\min,\min}$ are the subsets of the diagonal, and the bounded subsets are the  finite subsets.
 Let $G$ be a group.  
We consider $G_{\min,\min}$ in $G\BC$ using the left action along with the additional  right action of $G$. Then we have
 $$G_{\can,\min}\cong G_{\min,\min}//G$$ in $G\BC$, where "$\can$" indicates the canonical $G$-coarse structure on $G$ generated by the entourages $\{(g,h)\}$ for all pairs $g,h$ of elements of $G$.
 We consider the equivariant coarse $K$-homology functor $K\cX^{G}$ \cite{coarsek}.
Then \cref{gjiwioergwrf} provides the assembly map
$$\colim_{BG}K\cX^{G}(G_{\min,\min})\to K\cX^{G}(G_{\can,\min})\ .$$
 By making its domain and target explicit, it can be  identified with a version of the Davis-Lück/Baum-Connes  assembly map     (see \cite{kranz}) $$\colim_{BG}KU\to K(C_{r}^{*}(G))$$  for the family of the trivial subgroup.
 \hB 
 \end{ex}

\section{The squeezing space and the cone}\label{koppgwerefwerfwerf}
Recall the   category $\UBC$ of uniform bornological coarse spaces \cite{buen}, \cite{ass}. We start  with the 
   introduction of  the  squeezing space functor $$\Sq:\UBC\to \Fun(B\Z,\BC)\ .$$  
    The underlying bornological space of $\Sq(X)$ is the set $\Z \times X$ with
  the bornology generated by the subsets $F\times B$ for finite subsets $F$ of $\Z$ and bounded subsets $B$ of $X$.
  The coarse structure consists of all  
  sub-entourages of  entourages of the form \begin{equation}\label{rgesgerg}
(\diag(\Z)\times U)\cap W\ ,
\end{equation} where
$U$ is a coarse entourage of $X$ 
and $W$ is an entourage with the property that for every uniform entourage $V$ of $X$ there exists $n_{0}$ in $\Z$ 
such that for all $n$ in $\Z$ with $n\ge n_{0}$ we have $W_{n}\subseteq V$.  Here for an entourage $W$ on $\Z\times X$ we write $W_{n}:=W\cap (\{n\}\times X)^{2}$ for its restriction to the $n$th component of the squeezing space.
\begin{rem} The coarse structure of  $\Sq(X)$ is the hybrid coarse structure \cite{nw1}
on the bornological coarse space $\Z_{\min,\min}\otimes X$ associated  to the  additional compatible product  uniform  structure  and the big family $\Sq_{-} (X):=( (-\infty,n]\times X)_{n\in \nat}$; see \cite[Sec. 5.1]{buen}.
\hB\end{rem}The $\Z$-action on the squeezing space is given by the shift map
\begin{equation}\label{gertgtg5gtg}
t:\Sq(X)\to \Sq(X)\ , \quad (n,x)\mapsto (n+1,x)\ .
\end{equation}
A map $f:X \to Y$ of uniform bornological coarse spaces functorially induces a $\Z$-equivariant  map $$\Sq( f): \Sq(X)\to \Sq(Y) \ , \quad \Sq(f)(n,x):=(n,f(x))$$
of squeezing spaces.    
 \begin{ddd}\label{ugweogwregw} 
 We call the functor $X\mapsto  \Sq(X)$ described above the 
  squeezing space functor.
  \end{ddd}
  For $n$ in $\nat$ we let $\Sq_{\le n}(X) $  denote the subspace of $ \Sq(X)$ consisting of the points $(k,x)$ with $k\le n$. These subspaces generate the big family $\Sq_{-}(X)$ on $\Sq(X)$.
  The symbols $\Sq_{>n}(X)$ or $\Sq_{n}(X)$ have analogous interpretations.
\begin{rem}  Note that  $\Sq(X)$ does not belong to $\Z\BC\subseteq \Fun(B\Z,\BC)$  in general since it does not admit enough $\Z$-invariant coarse entourages.
 \hB \end{rem}

 %

\begin{ddd}\label{lpozhrtrger}
We define the geometric cone-at-infinity functor as the composition
$$\cO^{\infty}:\UBC\stackrel{\Sq}{\to} \Fun( B\Z,\BC)\stackrel{-//\Z}{\to} \BC\ .$$
\end{ddd}

\begin{rem}
The cone functors in \cite{buen}, \cite{ass} were defined using $\R$ instead of $\Z$. The canonical inclusion
$\Z\to \R$ induces a coarse equivalence between the present definition and the previous definitions in the references. 
 \hB
\end{rem}

The identity of underlying sets provides a map
$$\Sq(X) \to \Z_{\min,\min}\otimes c(X)$$
in $\Fun(B\Z,\BC)$, where $c$ is as in \eqref{fwerfvsdv}. 
Applying $-//\Z$  and using \cref{hokeprthgrtgertgt} we get the geometric cone boundary
\begin{equation}\label{erthertgrt}
\partial^{\geom}:\cO^{\infty}(X)\to \Z_{\can,\min}\otimes c(X)\ .
\end{equation}

%
%
%
%
%
%
%
%
%
%
%
%

\section{The motive of the cone}\label{gojpergewrfwefwf}

Recall the universal coarse homology theory  $\Yo$ from \eqref{gjiwejrgoewrijogwerf}.
If $X$ is a bornological coarse space with a big family $\cY$, then we set
$$\Yo(\cY):=\colim_{Y\in \cY} \Yo(Y)\ , \quad  \Yo(X,\cY):=\Cofib(\Yo(\cY)\to \Yo(X))\ .$$

Let $X$ be in $\UBC$.
The generating members $ \Sq_{\le n}(X)$ of the big family $\Sq_{-}(X) $  are preserved by the $\nat$-action given by down-shifts.
We therefore get an $\nat$-action  by down-shifts on $\Yo(\Sq(X),\Sq_{-}(X))$ in $\CM$.
Since all elements of $\nat$ act by equivalences, it extends to a $\Z$-action.
We indicate the twist of this action by  the sign character $\sign:\Z\to \{\pm 1\}$
by a superscript $\sign$.
Implicitly we use here  the existence of a corresponding functor $\sign:B\Z\to B\{\pm 1\} \to B\Aut_{\Sp}(S)$,
and the tensor structure $\CM\otimes \Sp\to \Sp$.

%

\begin{prop}\label{lptegtrtgetrgterg}
We have an equivalence (natural in $X$) 
\begin{equation}\label{fwefwefweerwf1}\colim_{B\Z} \Yo(\Sq(X),\Sq_{-}(X))^{\sign} \stackrel{\simeq}{\to} \Yo(\cO^{\infty}(X))
\end{equation} 
in $\CM$.    \end{prop}
\begin{proof}
The following Mayer-Vietoris argument is an abstraction of  \cite[Sec. 3]{Kitsios:2025aa}.
We consider
  the morphism
$$\iota^{\exp}: \Sq(X)\to \cO^{\infty}(X)\ , \quad  (n,x)\mapsto \left\{\begin{array}{cc} (n+1,x)&n<0\\ (2^{n},x) &n\ge 0  \end{array} \right.\ .$$
We further define the subset
$\cO^{odd}(X)\subseteq \cO^{\infty}(X)$ consisting of  
 $\cO_{\le 0}^{\infty}(X)$ and the intervals $[2^{n-1},2^{n}]\times X$ for all odd $n$ in $\nat$,
 and
 $\cO^{ev}(X)$ consisting of $\cO_{\le 1}^{\infty}(X) $ and the intervals $[2^{n-1},2^{n}]\times X$ for all  even $n\ge 2$ in $\nat$. The decomposition $(\cO^{odd}(X),\cO^{ev}(X))$ of $\cO^{\infty}(X)$ is coarsely excisive  and
 $\cO^{odd}(X)\cap \cO^{ev}(X)$ is precisely the image of $\iota^{\exp}$.
 We get a push-out square 
 $$\xymatrix{ \Yo(\Sq(X)_{\iota^{\exp}})\ar[r]\ar[d] &\Yo(\cO^{odd}(X)) \ar[d] \\\Yo(\cO^{ev}(X)) \ar[r] &\Yo(\cO^{\infty}(X)) } \ ,$$
 where the top horizontal and left vertical maps are induced by the corestrictions of 
 $\iota^{\exp}$, and  the subscript  indicates that we equip
the set with the coarse structure induced via $\iota^{\exp}$. We decompose
 $\Sq(X) $ into $\Sq^{odd}(X) $ consisting of all $(n,x)$ with $n< 0$ or $n$ odd in $\nat$, and
$\Sq^{ev}(X)$ consisting of $(n,x)$ with $n<0$ or $n$ even in $\nat$.
The restrictions of $\iota^{\exp} $ to  maps $\Sq^{odd/ev}(X)_{ \iota^{\exp}}\to  \cO^{odd/ev}(X)$
are coarse homotopy equivalences.

Taking the quotient by the big family $\cO^{\infty}_{-}(X)$ we 
 get a pushout
\begin{equation}\label{sfgsfdgewrgsffdgs}\xymatrix{ \Yo(\Sq(X)_{\iota^{\exp}},\Sq(X)_{\iota^{\exp}}\cap\cO^{\infty}_{-}(X))\ar[r]\ar[d] &\Yo(\cO^{odd}(X), \cO^{odd}(X)\cap\cO^{\infty}_{-}(X)) \ar[d] \\\Yo(\cO^{ev}(X),\cO^{ev}(X)\cap\cO^{\infty}_{-}(X)) \ar[r] &\Yo(\cO^{\infty}(X),\cO^{\infty}_{-}(X)) } \ .
\end{equation} 
Using $u$-continuity the upper-left corner can be written in the form
\begin{equation}\label{vsdfsdfvsdfvsdfv}  \Yo(\Sq(X) ,\Sq_{-}(X) )\simeq   \Yo(\Sq^{odd}(X) ,\Sq^{odd}_{-}(X) )\oplus   \Yo(\Sq^{ev}(X) ,\Sq_{-}^{ev}(X) )\ ,
\end{equation} 
and $\iota^{\exp}$ induces equivalences of these summands with the upper-right corner and the   lower-left corner of 
the square \eqref{sfgsfdgewrgsffdgs}.
In order to understand this note that $u$-continuity says that  $\Yo(\Sq(X)_{\iota^{\exp}},\Sq(X)_{\iota^{\exp}}\cap\cO^{\infty}_{-}(X))$ is the colimit
of $\Yo(\Sq(X)_{V},\Sq_{-}(X)_{V})$ over the  coarse entourages $V$ of $\Sq(X)_{\iota^{\exp}}$.
For every such coarse entourage $V$  
the components $\{n\}\times X$ become coarsely disjoint for sufficiently large $n$.
As we form the quotient by the components with small index anyway
we can can treat them all  as coarsely disjoint.

Since $\cO_{\le n}^{\infty}(X)$ is flasque for every $n$ in $\nat$ we have
$\Yo(\cO_{\le n}^{\infty}(X))\simeq 0$ and hence,
$\Yo(\cO_{-}^{\infty}(X))\simeq 0$.
Consequently the canonical map is an equivalence
$$ \Yo(\cO^{\infty}(X)) \stackrel{\simeq}{\to}\Yo(\cO^{\infty}(X),\cO_{-}^{\infty}(X))\ .$$

So, putting everything together, the square \eqref{sfgsfdgewrgsffdgs} induces a fibre sequence
\begin{equation}\label{gwereferfrewfw}\Yo(\Sq(X) ,\Sq_{-}(X) )\stackrel{F}{\to} \Yo(\Sq(X) ,\Sq_{-}(X) )\to \Yo(\cO^{\infty}(X))\ .
\end{equation} 
It remains to identify the map $F$.
In terms of the decomposition 
\eqref{vsdfsdfvsdfvsdfv}
it is given by a two-by-two matrix.
The diagonal entries are given by 
the restriction of $\iota^{\exp}$ and its (coarse homotopy) inverse, and therefore are the identities.
The maps
$$ \Yo(\Sq^{odd/ev}(X) ,\Sq^{odd/ev}_{-}(X) ) \to  \Yo(\Sq^{ev/odd}(X) ,\Sq^{ev/odd}_{-}(X) ) $$
are given by the bottom inclusions into the intervals for $n\ge 1$ composed with the 
inverse of the top inclusion. Hence they are induced by the restriction of the shift map $t$ in \eqref{gertgtg5gtg}.
We conclude that $F\simeq 1+t$. 
 
 The fibre sequence is the usual formula witnessing $\Yo(\cO^{\infty}(X))$ as  
 the colimit over $B\Z$  of $\Yo(\Sq(X) ,\Sq_{-}(X) )^{\sign}$. 

Note that the whole construction is natural in $X$.
\end{proof}

\begin{rem}
The   quotient map $\Sq(X)\to \Sq(X)//\Z\cong \cO^{\infty}(X)$
 induces  by \cref{gjiwioergwrf}  the assembly map
$$\colim_{B\Z} \Yo(\Sq(X))\to \Yo(\cO^{\infty}(X))\ .$$
We know that
$$\colim_{B\Z} \Yo(\Sq_{-}(X))\to \colim_{B\Z} \Yo(\Sq(X))\to \Yo(\cO^{\infty}(X))$$ vanishes 
since already  
  $ \Yo(\Sq_{-}(X))\to   \Yo(\cO^{\infty}(X))$ vanishes.
  So we get an assembly  morphism
  $$\colim_{B\Z} \Yo(\Sq(X),\Sq_{-}(X))\to \Yo(\cO^{\infty}(X))\ .$$
  One could ask whether this morphism is an equivalence too.  
   \hB \end{rem}
%
%
%
%


The following result can be shown by the same argument.
Assume that $X$ is in $\Fun(BG,\UBC)$.  
\begin{prop} \label{lkprhrhertgertgertge}
In $\CM$ we have a natural equivalence 
\begin{equation}\label{fwefwefweerwf}\colim_{B\Z} \Yo(\Sq(X)//G,\Sq_{-}(X)//G)^{\sign} \stackrel{\simeq}{\to} \Yo(\cO^{\infty}(X)//G)
\end{equation} 
in $\CM$.
\end{prop}
 
 In this note we will use the fibre sequence
 \begin{eqnarray}  \lefteqn{
 \Sigma^{-1}\Yo(\cO^{\infty}(X)//G)\stackrel{\delta}{\to}   \Yo(\Sq(X)//G,\Sq_{-}(X)//G)}\hspace{2cm}&& \nonumber \\&& \stackrel{1+t}{\to}
  \Yo(\Sq(X)//G,\Sq_{-}(X)//G)\stackrel{\Yoctr( \iota^{\exp}) }{\to}\Yo(\cO^{\infty}(X)//G)\ .\label{rewgerfewrfwe}
\end{eqnarray}  

The geometric cone boundary \eqref{erthertgrt} together with the usual equivalence $$\Yo(\Z_{\can,\min}\otimes -)\simeq \Sigma \Yo(-):\BC\to \CM$$  provide the cone boundary; a natural transformation
\begin{equation}\label{grewfwerfwef}
\partial^{\cone}:\Yo(\cO^{\infty}(-))\to \Sigma \Yo(c(-)):\UBC\to \CM\ ,
\end{equation}
where $c$ is as in \eqref{fwerfvsdv}.
We use the same notation for the analogously defined transformation 
$$\partial^{\cone}:\Yo(\cO^{\infty}(-)//G)\to \Sigma \Yo(c(-)//G):\Fun(BG,\UBC)\to \CM\ .$$


\section{Branched coarse coverings}\label{joopwbergregrfwrf}
 
For a $G$-set $X$   we consider the $G$-set
$G\times X$ with the diagonal $G$-action (involving  left-multiplication on $G$) and the additional $G$-action by right-multiplication on $G$. We consider the map
$$f:G\times X\to X\ , \quad f(g,x):=g^{-1}x\ .$$
It intertwines the additional $G$-action with the given $G$-action on $X$, and the diagonal $G$-action with the trivial $G$-action on $X$.

For   an entourage $W$    of $X$ we define the  entourage 
$$W_{G}:=f^{-1}(W)\cap \pr_{G}^{-1}(\diag(G)) $$
%
of  $G\times X$
which is invariant under the diagonal $G$-action and satisfies 
 satisfies $f(W_{G})=W$. 
 If $X$ has a  coarse structure, then we equip $G\times X$ with the  coarse structure 
 generated by the entourages $W_{G}$ for all coarse  entourages $W$ of $X$.
  If the coarse structure  on $X$ is preserved by the $G$-action, then the additional $G$-action preserves
  the coarse structure on $G\times X$. 

If $X$ has a bornology, then we equip $G\times X$ with the bornology generated by
the subsets $F\times B$ for finite subsets $F$ of $G$ and bounded subsets $B$ of $X$.
If the bornology on $X$ is preserved by $G$, then the map $f$ is bornological with locally finite fibres, but not proper in general.
 
If $X$ is a   bornological coarse space with $G$-action, then 
the structures on $G\times X$  defined above  turn  this set into a   $G$-bornological coarse space,  which we will denote by $G\ltimes X$.  
 We thus get a functor
 $$G\ltimes -:\Fun(BG,\BC)\to \Fun(BG,G\BC)\ .$$

\begin{ex}
If $X$ is in $G\BC$, then $f:G\ltimes X\to X$ is a bounded covering in the sense of \cite{trans}.
Indeed, the map $(g,x)\to (g,g^{-1}x)$ induces an isomorphism
$G\ltimes X\cong G_{\min,\min}\otimes X$ under  which $f$ becomes the projection onto the second factor. \hB
\end{ex}

Let $X$ be in $\UBC$.

\begin{ddd}\label{okbepgrbkeprbegfdb}
A uniform scale on $X$ is a  cofinal family $(U_{n})_{n\in \nat}$ in the uniform structure of $X$.
\end{ddd}

Note that the existence of a uniform scale is a non-trivial condition on the uniform structure.

We say that a  second uniform  scale  $(V_{n})_{n\in \nat}$ dominates the first if
$U_{n}\subseteq V_{n}$ for all $n$ in $\nat$.
\begin{ex}\label{kohprethegtgtrge}
If the uniform structure of $X$ comes from a metric, then $X$ admits a uniform scale. One can take  
$(V_{n^{-1}})_{n\in \nat}$, where $V_{r}$ is the metric entourage of width $r$ and we interpret 
$V_{\infty}:=X\times X$.

Every uniform scale $(U_{n})_{n\in \nat}$ of $X$ is dominated by a uniform scale of the form
$(V_{r(n)})_{n\in \nat}$, where $r:\nat\to (0,\infty]$ is a function such that $\lim_{n\to \infty} r(n)=0$.
We can take $r(n):=\inf\{r\in (0,\infty]\mid U_{n}\subseteq V_{r}\}+1$.
\hB
\end{ex}

 Let $X$ be in $\Fun(BG,\UBC)$. 
 \begin{ddd}\label{kohpertgetrgeg1}
 A uniform scale $(V_{n})_{n\in \nat}$ is called a Lipschitz scale for the $G$-action  if for every $g$ in $G$ there exists  $k$ in $\nat$ (the Lipschitz constant of $g$) such that for every $n$ and $l$ in $\nat$  we have $gV^{l}_{n}\subseteq V_{n}^{lk}$.
 \end{ddd}

    \begin{ddd}\label{kohpertgetrgeg}
   The action of $G$ on $X$ is called Lipschitz if 
   every uniform scale   on $X$ is dominated by a Lipschitz scale. 
    
   \end{ddd}
   

\begin{ex}
If the uniform structure of $X$ comes from a metric   and $G$ acts by Lipschitz maps,
then the action of $G$ on $X$ is Lipschitz.  In order to construct a Lipschitz scale 
dominating a given uniform scale,   use the second part of \cref{kohprethegtgtrge}. \hB
\end{ex}

Recall from \cref{koppgwerefwerfwerf} that the coarse structure of $\Sq(X)$  contains a cofinal set of entourages of the form
$V=(\diag(\Z)\times U)\cap W$ (see \eqref{rgesgerg})  for a coarse entourage $U$ of $X$  that is also uniform,  and
an entourage $W$ of $\Z\times X$ such  that  every given  uniform entourage of $X$
 contains the  $n$-th component $W_{n}$ 
provided $n$ is sufficiently large. 
Recall that $\Sq(X)_{V}$ denotes the bornological coarse space obtained from $\Sq(X)$ by  
replacing the original coarse structure by the one generated solely by the single entourage $V$.
While   $G$ acts on $\Sq(X)$ by the functoriality of the squeezing space functor,
we need additional conditions to ensure that $G$ also acts on the bornological coarse spaces  $\Sq(X)_{V}$ for a cofinal set of entourages $V$.
\begin{ass}\label{hepretgetg} \mbox{}\begin{enumerate}
 \item \label{gwpokrgwrfwr}
The uniform structure of $X$ admits a  uniform scale.\item\label{gwpokrgwrfwr0}
We   assume that $X$ has an  underlying $G$-bornological coarse space.
\item \label{gwpokrgwrfwr11} We assume that the $G$-action on $X$ is Lipschitz.\end{enumerate}
 \end{ass}
By \cref{hepretgetg}.\ref{gwpokrgwrfwr}
we can assume after enlarging $V$ that $(V_{n})_{n\in \nat}$ is a uniform scale, and that $V_{n}=U$ for $n\le 0$ for some coarse entourage $U$   of $X$ as in \eqref{rgesgerg} that is also uniform. Let us call such an entourage $V$ uniformly wide.
 By \cref{hepretgetg}.\ref{gwpokrgwrfwr0} we can assume further that $U$ is $G$-invariant.  In this case by \cref{hepretgetg}.\ref{gwpokrgwrfwr11} we can enlarge $V$ further such that  $gV^{l}\subseteq V^{kl}$ for every $l$ in $\nat$, where $k$
 is the Lipschitz constant of $g$ and does not depend on $l$.

 \begin{ddd}\label{herjtogitrjgoetrgrtegetrg} An entourage  $V$ of $\Sq(X)$
is called $G$-adapted   if it is symmetric,  uniformly wide,  and for every $ g$ in $G$ there exists $k$ in $\nat$ such that for every $l$ in $\nat$ we have   $gV^{l}\subseteq V^{lk}$.
 \end{ddd}
   We get the following assertion:
  \begin{kor}\label{jophrthetgetgetg}
Under \cref{hepretgetg} the bornological coarse space $\Sq(X)$ admits  a cofinal set of $G$-adapted  coarse entourages of $\Sq(X)$.    \end{kor}

  \begin{kor}\label{jophrthetgetgetg1}
  If $V$ is a $G$-adapted coarse entourage of $\Sq(X)$, then $G$ acts on $\Sq(X)_{V}$.
  \end{kor}

    \begin{ddd}\label{jkrtpzrtzhjrthrtzhrth}
 We say that the action is uniformly free if for every  $g$ in $G$ with $g\not=e$ there exists a uniform entourage  $U$ of $X$ with $U\cap \{(gx,x)\mid x\in X\}=\emptyset$. 
   \end{ddd}
   
   \begin{ex}
If $X$ is compact and the action of $G$ is free in the set-theoretic sense, then it is uniformly free.
\hB
\end{ex}

We recall   the notion of a branched coarse $G$-covering \cite[Def. 2.1]{Bunke:2025aa}.
Let $X$ be a $G$-bornological coarse space, let $Y$ be a bornological coarse space equipped with a big family, and let    $f:X\to Y$ be a $G$-equivariant  map between the underlying sets, where $Y$ is equipped with the trivial $G$-action.
Furthermore, let  $\cZ$ be a big family in $Y$.
\begin{ddd}\label{jijhitrojhortrgrtgertgeg}The pair 
 $(f:X\to Y,\cZ)$ is a branched coarse $G$-covering with respect to $\cZ$ if $f$ is controlled, bornological, has locally finite fibres,
 and if there exists a $G$-invariant entourage $P$ on $X$ (the connection)
 with the following properties:
 \begin{enumerate}
 \item\label{kohpezhrtgergeg} For every coarse entourage $V$ of $Y$ there exists a member $Z$ in $\cZ$
 such that for every $(y',y)\in V\cap (Y\times (Y\setminus Z))$ and $x$ in $f^{-1}(y)$ there exists a unique $x'$ in $f^{-1}(y')
$
such that $(x',x)\in P$, and the subset
$P\circ (f\times f)^{-1} (V\cap (Y\times (Y\setminus Z)))$ is a coarse entourage of $X$.
\item \label{kohpezhrtgergeg1} For every coarse entourage $U$ of $X$ there exists a member $Z$ in $\cZ$ such that $U\cap (X\times f^{-1}(Y\setminus Z))\subseteq P$.
 \end{enumerate}
     \end{ddd}

The following is an analogue of \cite[Prop. 3.10]{Sawicki_2020}. 

Let $X$ be in $\Fun(BG,\UBC)$. 
\begin{prop} \label{uoiwgrgwegwerfwfwerf} We assume:
\begin{enumerate}
\item  $G$ is countable.
\item  The $G$-action on $X$ is uniformly free. 
\end{enumerate}
If $V$ is a $G$-adapted coarse entourage of $\Sq(X)$, then 
the map
 $$f:(G\ltimes \Sq(X)_{ V}) //G\to  \Sq(X)_{  V}//G\ , \quad (g,n,x)\mapsto (n,g^{-1}x) $$ is a branched
 coarse $G$-covering with respect to the big family $ ((-\infty,n]\times X)_{n\in \nat}$.
\end{prop}
\begin{proof}
Since $G$ acts on $\Sq(X)_{V}$ by \cref{jophrthetgetgetg1}, it also acts on $G\ltimes \Sq(X)_{V}$ by functoriality of $G\ltimes -$ so that $(G\ltimes \Sq(X)_{ V}) //G$ is well-defined.
As observed above, the map $f$ is bornological, controlled and has locally finite fibres.

We can fix a sufficiently slowly growing function $k:\Z\to \Z$ such that $\lim_{n\to \infty} k(n)=\infty$ and that
$(V_{n}^{k(n)})_{n\in \Z}$ is still a uniform  scale  on $X$.
For every $n$ we let $B'_{n}$ be the subset of $g$ in $G$ such that the conditions 
$(y,x)\in V_{n}^{k(n)}$ and $(gy,x)\in V_{n}^{k(n)}$  are incompatible for all pairs $x,y$ in $X$.
It follows from uniform freeness and the cofinality of $ (V_{n}^{k(n)})_{n\in \Z}$ that the increasing family 
$(B'_{n})_{n\in \Z}$ exhausts $G\setminus \{1\}$.  Using the countability of $G$ we now choose an increasing  family
$(B_{n})_{n\in \Z}$  still exhausting $G\setminus \{1\}$ such that $B_{n}\subseteq B_{n}'$ for every $n$ in $\Z$ and for every $l$ in $G$ we have
$lB_{n}l^{-1}\subseteq B'_{n}$ for sufficiently large $n$.

Since the components $\{n\}\times X$ of $\Sq(X)$ are coarsely disjoint, we can define the connection \begin{equation}\label{erfwerfwffrfreferf}P:=\bigcup_{n\in \Z} P_{n}
\end{equation} by providing its components.
  For any subset $B$ of $G$
we consider the $G$-invariant entourage \begin{equation}\label{ertkhertgrtgge}
U_{B}:=\{(gb,g)\mid g\in G,  b\in B\} \ .
\end{equation} of $G$. We then set 
$$P_{n}:=   (U_{B_{n}}\times \diag(X))  \circ  (V_{n}^{k(n)})_{G}  
$$ on $G\times X$
for all $n$ in $\nat$.

We now verify  Condition \ref{jijhitrojhortrgrtgertgeg}.\ref{kohpezhrtgergeg}.
For every finite subset $F$ of  $G$, we consider the entourage $U_{F}$   in the coarse structure of $\Sq(X)//G$  defined by  $$U_{F}:=\{((n,s^{-1}x),(n,x))\mid n\in \Z\ , x\in X\ , s\in F\}\ .$$  
 Using  that 
 $V$ is $G$-adapted,  one checks that the family of entourages  $(U_{F}\circ V^{k})_{F\subseteq G,\   k\in \nat}$ is cofinal in the coarse structure of $\Sq(X)_{V}//G$.  

We fix  $R$  in $\nat$ and a finite subset $F$ of $G$. 
We choose $n_{0}$ in $\Z$ such that $R\le k(n)$, $F\subseteq B_{n}$ and $\bigcup_{f\in F}fB_{n}f^{-1}\subseteq B_{n}'$ for all $n > n_{0}$.
We will show that the parallel transport  at scale $U_{F}\circ V^{R}$ is well-defined on
the complement $(n_{0},\infty)\times X$ of
  the member $(-\infty,n_{0}]\times X$ of the big family $ ((-\infty,n]\times X)_{n\in \nat}$.

Assume that    $((n, x),(n,y))\in U_{F}\circ V^{R}$ and $n>n_{0}$.
Then  there exists $s$ in $F$  such that $(sx,y)\in V_{n}^{R}$.
  Let $(l,n,ly)$ be a preimage under $f$ of $(n,y)$. 
Then $(ls,n,lsx)$ is a preimage of $(n,x)$ and
$((ls,n,lsx),(l,n,ly))\in P$. A different preimage of $(n,x)$   could be written
as $(lsb,n,lsbx)$ with $b\not=e$.  If $((lsb,n,lsbx), (l,n,ly))\in P_{n} $,  then  $sb\in B_{n}$,   $(sbx,y)\in V_{n}^{k(n)}$ and $(sx,y)\in V_{n}^{k(n)}$. Since $sbx=(sbs^{-1} )sx$ and $sbs^{-1}\in B_{n}'$ this is impossible.
This finishes the verification of  unique path-lifting on $ (n_{0},\infty)\times X$ at scale $U_{F}\circ V^{R}$.
 
 Note that $P_{\ge n}$ itself its not a coarse entourage of $(G\ltimes \Sq(X)_{\ge n})//G$ since $n\mapsto k(n)$ is unbounded.
  But  the above calculations show that
 $(f^{-1}(U_{F}\circ V^{k}) \cap P)_{ \ge n} \subseteq \bigcup_{n} (U_{B_{n}}\times \diag(X))\circ (V_{n}^{k} )_{G}$
 for large $n$ which is a coarse entourage of   $(G\ltimes \Sq(X)_{\ge n})//G$.
 This  finishes the verification of  Condition \ref{jijhitrojhortrgrtgertgeg}.\ref{kohpezhrtgergeg}.

 For every $i$ in $\nat$ 
  and every  finite subset $F$ of $G$ we have $((U_{F}\times \diag(X))\circ (V^{i})_{G})_{n}\subseteq P_{n}$ provided
  $n$ is so large that $F\subseteq B_{n}$ and $k(n)\ge i$. This verifies the Condition  \ref{jijhitrojhortrgrtgertgeg}.\ref{kohpezhrtgergeg1}. 
\end{proof}
A similar  proof shows:

 \begin{prop}\label{jiohnergvfvfdgvbdfbdfgb}We assume:
\begin{enumerate}
\item  $G$ is countable.
\item  The $G$-action on $X$ is uniformly free. 
\item \cref{hepretgetg}
\end{enumerate}

Then there exists a cofinal set of entourages $V$ of $\cO^{\infty}(X)$ such that $G$ acts on $\cO^{\infty}(X)_{V}$ and 
 the map
$$f:(G\ltimes \cO^{\infty}(X)_{V})//G\to \cO^{\infty}(X)_{V}//G$$
is a branched coarse $G$-covering with respect to the big family  $   ((-\infty,n]\times X)_{n\in \nat}$. \end{prop}

 \section{Motivic calculations using transfers}\label{kopggwregweg}
 
 The notion of a coarse homology theory with transfers was introduced in \cite{trans}.
 Many interesting coarse homology theories, in particular those of $K$-theoretic nature, extend to coarse homology theories with transfers; see \cite{trans} and \cite{unik} for algebraic  coarse  $K$-homology theories, and \cite{coarsek} for topological coarse $K$-homology.  
 
 The domain of a  coarse homology theory with transfers   is 
  the two-category  $\BC_{\tr}$, whose objects are bornological coarse spaces,  whose one-morphisms are spans
 $$\xymatrix{&W\ar[dl]_{g}\ar[dr]^{f}&\\X\ar@{..>}[rr]&&Y}$$
 where $f$ is a morphism in $\BC$ that is in addition bornological and $g$ is a bounded covering, and whose 
 two-isomorphisms are isomorphisms of such spans.
 We have a canonical functor
 $$\iota:\BC\to \BC_{\tr}\ , \quad (f:X\to Y)\mapsto \xymatrix{&X\ar@{=}[dl]\ar[dr]^{f}&\\X\ar@{..>}[rr]&&Y}\ .$$
 By definition, a  coarse homology theory with transfers
 is
 functor   $E:\BC_{\tr}\to \cC$ to a cocomplete stable $\infty$-category such that  $E\circ \iota:\BC\to \cC$ is a coarse homology theory.

If
$f:X\to Y$ is a bounded covering, then we write $f^{*}:Y\to X $ in $\BC_{\tr}$ for the span 
 $$\xymatrix{&X\ar[dl]_{f}\ar@{=}[dr]&\\Y\ar@{..>}[rr]^{f^{*}}&&X}\ .$$
 The span
 $$\xymatrix{&X\ar[dl]_{\incl}\ar@{=}[dr]&\\X\sqcup Y\ar@{..>}[rr]^{\pr_{X}}&&X}$$
 represents the projection onto the component $X$ of the coproduct. The projections $\pr_{X}$ and $\pr_{Y}$ together 
witness that  the coproduct  is also the product. 
 The category $\BC_{\tr}$ is therefore semi-additive and  hence enriched in commutative monoids. We add spans by forming the disjoint union of their middle entries.

For formal reasons, we have a universal coarse homology theory with transfers \eqref{vsdfcsdvdfvsdfver}.
  
  \begin{ex}
  Let $X$ be in $\Fun(BG,\UBC$). For the moment we forget the $\Z$-action and consider $\Sq(X)//G$ as an object of $\BC$.
Using excision we get for every $n$ in $\Z$ a  projection map
$$\pr_{n}:\Yo(\Sq(X)//G)\simeq  \Yo(\Sq_{\not=n}(X)//G) \oplus \Yo( \Sq_{n}(X)//G)\stackrel{\pr}{\to}  \Yo( X//G
)$$ in $\CM$, where we implicitly used the isomorphism $\Sq_{n}(X)//G  \cong X//G$.
In the context with transfers this projection onto the $n$th component already exists in $\BC_{\tr}$ and is represented by the span
 $$\xymatrix{&X//G\ar[dl]_{x\mapsto (n,x)}\ar@{=}[dr]&\\\Sq(X)//G\ar@{..>}[rr]^{\pr_{n}}&&X//G}\ .$$
  \hB
  \end{ex}

 \begin{ex}\label{oijoegwerferfwref}
We use transfers in order to   construct 
 a map
 \begin{equation}\label{rferffwe} \phi:\Yoctr(X_{\disc}//G)\to \lim_{B\Z}  \Yoctr(\Sq(X)//G)\ .
\end{equation}
We consider $X$ in $\Fun(BG,\UBC$) and let 
  $X_{\disc}$ denote $X$ with the discrete coarse structure. The    identity of underlying sets   is a morphism
$$\kappa:\Z_{\min,\min}\otimes X_{\disc}//G\to \Sq(X)//G\ .$$
In $\BC_{\tr}$ we have a span
$$\xymatrix{&\Z_{\min,\min}\otimes X_{\disc}//G\ar[dl]_{\pr_{X//G}}\ar[dr]^{\kappa}&\\X_{\disc}//G \ar@{..>}[rr]^{\phi} &&\Sq(X)//G}\ .$$
The isomorphism of spans
 $$\xymatrix{&\Z_{\min,\min}\otimes X_{\disc}//G\ar[dl]_{\pr_{X//G}}\ar[dd]_{\cong}^{(n,x)\mapsto (n+1,x)}\ar[dr]^{t\circ \kappa}&\\X_{\disc}//G \ar@{..>}@/^0.5cm/[rr]^(0.35){t\circ \phi}\ar@{..>}@/^-0.5cm/[rr]_(0.35){\phi} &&\Sq(X)//G \\&\ar[ul]^{\pr_{X//G}}\Z_{\min,\min}\otimes X_{\disc}//G\ar[ur]_{\kappa}&}$$
 witnesses an isomorphism between
 $t\circ \phi$ and $\phi$.
 We have therefore constructed a $\Z$-equivariant refinement  \eqref{rferffwe} also denoted by $\phi$. 
 
  Since a limit over $B\Z$ is a finite limit, and we work with stable $\infty$-categories, it is preserved by 
   coarse homology theories $E$ with transfers considered as colimit-preserving functors $E:\CMctr \to \cC$. We therefore have
 a transfer map  $$E(\phi):E(X_{\disc}//G)\to \lim_{B\Z} E(\Sq(X)//G)\ .$$
In particular, if $\cC=\Sp$ and we have a class $e$ in $\pi_{
*}E(X_{\disc}//G)$, then $\phi_{*}e$ in $\pi_{*
}E(\Sq(X)//G)$ is a $\Z$-invariant class; i.e., it satisfies $$t_{*}\phi_{*}(e)=\phi_{*}(e)\ .$$ 
The components of this class are given by the image of $e$ under the canonical map $E(X_{\disc}//G)\to E(X//G)$.
 \hB
 \end{ex}

 Recall that a rig is a set with  the same structures as a (possibly non-commutative) ring, but the axiom of the existence of inverses for the addition is dropped. Typical examples of rigs are $\nat$ or $\End_{\cC}(C)$ for objects $C$ in a semi-additive $1$-category.
 
 The $\Z$-action on the rig $\prod_{\Z}\nat$ by shifts gives rise to
 a crossed product rig $\prod_{\Z}\nat\rtimes \Z$. 
 In the following, we construct for every object $X$ in $\Fun(BG,\UBC)$  
  a homomorphism of rigs \begin{equation}\label{gwerojofpwerfwerf}a:\prod_{\Z}\nat\rtimes \Z\to \End_{\ho(\BC_{\tr})}(\Sq(X)//G)\ .
\end{equation} 
  
In $\End_{\BC_{\tr}}(\Sq(X)//G)$ we consider the full symmetric monoidal subcategory \begin{equation}\label{erowgwopergwefwerf}R(X)\subseteq \End_{\BC_{\tr}}(\Sq(X)//G)
\end{equation} generated by 
the spans 
$$\xymatrix{& \Sq(X)//G\ar[dr]^{(n,x)\mapsto (n+k,x)} \ar@{=}[dl] &\\ \Sq(X)//G \ar@{..>}[rr]^{t^{k}}  &&  \Sq(X)/G}\ .$$
for $k$ in $\Z$ (the shifts), and the spans
$$\xymatrix{&\Sq^{\bar k}(X)//G\ar[dr]^{f_{\bar k}}\ar[dl]_{f_{\bar k}}&\\ \Sq(X)//G \ar@{..>}[rr]^{ a(\bar k) }&&  \Sq(X)/G}\ .$$  for 
 $\bar k=(k_{n})_{n\in \Z}$   in  $\prod_{\Z}\nat$.  
Here, the underlying set of  $ \Sq^{\bar k}(X)//G$ is $\bigsqcup_{n\in \Z} \{1,\dots,k_{n}\}\times X$ and 
$f_{\bar k}$ is the canonical projection 
$$\bigsqcup_{n\in \Z} \{1,\dots,k_{n}\}\times X\to \bigsqcup_{n\in \Z} \{*\}\times X\cong \Z\times X\ .$$
The bornological and coarse structures are induced from those of $\Sq(X)//G$ in the canonical way.

One easily checks  that the rig of isomorphism classes on $R(X)$
is isomorphic to $\prod_{\Z}\nat\rtimes \Z$. So \eqref{gwerojofpwerfwerf} is induced by the inclusion  \eqref{erowgwopergwefwerf} on the level of homotopy categories.

\begin{rem}
If $X$ is uniformly and coarsely connected, then 
$R(X)\simeq  \End_{\BC_{\tr}}(\Sq(X)//G)$. \hB
\end{rem} 

Note that
$\End_{\ho(\CMtr)}(\Yoctr(\Sq(X)//G))$ is  a ring.
We therefore get a factorization
\begin{equation} 
\xymatrix{ \prod_{\nat} \nat\rtimes \Z\ar[r]^-{a}\ar[d]  &  \End_{\ho(\BC_{\tr})}(\Sq(X)//G)\ar[d]^{\Yoctr} \\   \prod_{\nat} \Z \rtimes \Z\ar@{..>}[r]^-{a} &\End_{\ho(\CMctr)}(\Yoctr(\Sq(X)//G)) } \ ,
\end{equation} denoted by abuse of notation also by $a$, 
 where the left vertical map is the additive group-completion.

 Assume that $E:\BCtr\to \Sp$ is a coarse homology theory  with transfers and that
 $x$  is a class in $\pi_{*}E(\Sq(X)//G)$. Recall that $t_{*}x:=a(t)_{*}x$ denotes the result of applying the shift to $x$.
 If $\bar k=(k_{n})_{n\in \Z}$ is in $\prod_{\Z}\Z$, then we use the notation
 $a(\bar k)_{*}x:=\bar k\cdot x$.

%
%

\begin{rem}
There is no need to pass to homotopy categories. For every $X$ in $\Fun(BG,\UBC)$ we have an
   equivalence of symmetric monoidal  rig-space $R(X)\simeq \prod_{\Z} \Fin^{\simeq}\rtimes \Z$, see
  \cite{Gepner_2015}. 
Group-completing the right-hand side we get a ring spectrum  $\bA:=(\prod_{\Z} \Fin^{\simeq}\rtimes \Z)^{\grp} $. Then for every object $X$ in $\Fun(BG,\UBC)$ we  have a homomorphism of ring spectra
  $$\bA\to \End_{\CMctr}(\Yoctr(\Sq(X)//G))\ .$$
 Passing to $\pi_{0}$, we get $\pi_{0} \bA\cong \prod_{\Z}\Z\rtimes \Z$ and recover $a$.
 Using the Barratt-Priddy-Quillen equivalence, $(\Fin^{\simeq})^{\grp}\simeq S$ we get a map
 $\bA\to  \prod_{\Z}S\rtimes \Z$ which we expect to be an equivalence. 
We have $ (\prod_{\Z} \Fin^{\simeq})^{\grp}\simeq  \prod_{\Z}S$ by \cite[Prop. 7.2.5]{unik},
but at the moment we have no reference for the fact that group completion commutes with
forming the crossed product.
 \hB
\end{rem}

The homomorphism $a$ induces a ring homomorphism
$$\bar a: \frac{\prod_{\nat}\Z}{\bigoplus_{\nat}\Z}\rtimes \Z \to \End_{\ho(\CMctr)}(\Yoctr( \Sq(X)//G,\Sq_{-}(X)//G))\ .$$

We end this section with an application to the construction of coarse homology classes on the cone $ \cO^{\infty}(X)//G$.
Let $E$ be a spectrum-valued  coarse homology theory with transfers and $x$ be in $\pi_{*}E(\Sq(X)//G,\Sq_{-}(X)//G)$.  Assume that $$t_{*}x= \phi \cdot x$$ for some
$\phi$ in $\frac{\prod_{\nat}\Z}{\bigoplus_{\nat}\Z} $.
Then 
we have
$$(1+t) (\psi \cdot x)=0 \ , \qquad x=(1+t) \kappa \cdot x\ ,$$ where
$\psi$ and $\kappa$ in $\frac{\prod_{\nat}\Z}{\bigoplus_{\nat}\Z}$ solve \begin{equation}\label{iojweorgwoerf}\psi+(t\psi) \cdot \phi=0 \ , \quad (\kappa-\bar 1)+ (t\kappa)\cdot \phi=0
\end{equation}
and $\bar 1$ in $ \frac{\prod_{\nat}\Z}{\bigoplus_{\nat}\Z}$ is the  unit.
The second equation are always solvable recursively.
This implies that  
$\iota^{\exp}_{*} x=0$ for $\iota^{\exp}$ as in \eqref{rewgerfewrfwe}. As a result we see that 
the class $x$ is not useful to produce interesting classes in $\pi_{*}E(\cO^{\infty}(X)//G)$ via $\iota^{\exp}$, see also \cref{hkepetgerhrthergrtg}.
A similar observation has already been made in \cite[Sec. 4.1]{Kitsios:2025aa}.

By solving the  first equality in \eqref{iojweorgwoerf} we can find   new classes
$\psi \cdot x$ in the kernel of $1+t$   and therefore in the image of $$\delta: \pi_{*+1}E( \cO^{\infty}(X)//G)\to
  \pi_{*}E(\Sq(X)//G ,\Sq_{-}(X)//G )$$ with $\delta$ as in 
\eqref{rewgerfewrfwe}.
  Such a  class is non-trivial provided $x$ is so and
  $\psi$ is invertible.  The recursive equation for $\psi$ is
  $\psi_{n}:=- \psi_{n-1} \phi_{n}$ which must hold for sufficiently large $n$. 
   Since $\phi_{n}$ is eventually non-zero, we see that eventually non-zero solutions for $\psi$ exist as well.
  
  If $t_{*}x=x$, then we can take  $\psi=[(\pm 1)]:=[((-1)^{n})_{n\in \nat}]$.
  Again, this possibility of constructing classes on the cone has already been observed in  \cite[Sec 4.2]{Kitsios:2025aa}.
  Note that \cref{oijoegwerferfwref} can be used to produce classes $x$ with this property.

  \begin{ex}\label{hkepetgerhrthergrtg}
  In this example, we recover the motivic analogue of  \cite[Thm. B]{Kitsios:2025aa} for the cone at $\infty$. The proof is essentially a version of 
   \cite[(4.1)]{Kitsios:2025aa}.
   We consider the quotient map
   $$\iota:\Sq(X)//G \to \cO^{\infty}(X)//G$$
   induced by the identity of underlying sets.
   Below we implicitly use the equivalence   $$\Yoctr(\cO^{\infty}(X)//G)\stackrel{\simeq}{\to}\Yoctr(\cO^{\infty}(X)//G, \cO_{-}^{\infty}(X)//G)$$ for the second map in \eqref{erwfwerferwfwf}.
   The map $c$  in \eqref{erwfwerferwfwf} is the canonical map.

   \begin{prop}\label{okheprthertgetrge}
   The map
  \begin{equation}\label{erwfwerferwfwf}
\lim_{B\Z}\Yoctr(\Sq(X)//G,\Sq_{-}(X)//G)\stackrel{c}{\to} \Yoctr(\Sq(X)//G,\Sq_{-}(X)//G)\stackrel{\Yoctr(\iota)}{\to} \Yoctr(\cO ^{\infty}(X)//G)
\end{equation}  is  trivial. \end{prop}
   \begin{proof}
One checks, using a coarse homotopy, that $$\Yoctr(\iota)\simeq \Yoctr( \iota^{\exp})\circ  \psi  :\Yoctr(\Sq(X)//G,\Sq_{-}(X)//G)\to\Yoctr(\cO ^{\infty}(X)//G)\ ,$$ where $\psi=[(\psi_{n})_{n\in \Z}]$ in $\frac{\prod_{\nat}\Z}{\bigoplus_{\nat} \Z}$ with $\psi_{n}= 2^{n-1} 
 $ for large $n$. 
  We can solve $(1+t) \phi = \psi$ for $\phi$ in $\frac{\prod_{\nat}\Z}{\bigoplus_{\nat} \Z}$.
  Then $(1+t)\phi c\simeq \psi c$ and therefore 
  $\Yoctr(\iota) c\simeq \Yoctr( \iota^{\exp}) \psi  c\simeq \Yoctr( \iota^{\exp}) (1+t) \phi  c\simeq 0$ by \eqref{rewgerfewrfwe}.
  \end{proof}   
  \end{ex}
  
  \begin{rem}
  The map $\iota$ restricts to a map
   $$\iota_{\ge 0}:\Sq_{\ge 0}(X)//G \to \cO_{\ge 0}^{\infty}(X)//G\ .$$
   Using an analogous argument as above  and the exponential Mayer-Vietoris decomposition 
of $\cO_{\ge 0}^{\infty}(X)//G$ one can show that 
   $$\lim_{B\Z}\Yoctr(\Sq(X)//G )\stackrel{c}{\to} \Yoctr(\Sq(X)//G )\stackrel{\Yoctr(\iota_{\ge 0})\circ \Yoctr(\pr_{\ge 0})}{\to}\Yoctr( \cO_{\ge 0}^{\infty}(X)//G)$$
   is trivial. 
   Applied to the coarse homology theory with transfers $K\cX$
  and the class $[p]$ in $$\lim_{B\Z} K\cX(   \Sq(X)//G,\Sq_{-}(X)//G)$$ for $p$ as in \eqref{whtgwergergws}, we precisely obtain 
     \cite[Thm. B]{Kitsios:2025aa} with the same argument as in    \cite[(4.1)]{Kitsios:2025aa}.
\hB
  \end{rem}

\uli{bis hier}

%
%
%
%

%
%
%
%

\section{The motivic  coarse assembly map}\label{kohperhretgegtre}

Recall the universal homological functors $\Yo$, $\Yo^{\strg}$ and $\Yo\cB$ from \eqref{gjiwejrgoewrijogwerf} and  \eqref{werfeferfsfvfd}.   
Since any strong coarse homology theory is, in particular, a coarse homology theory, we have a unique  colimit-preserving   comparison map $\cp$ fitting into  the commutative triangle
$$\xymatrix{&\BC\ar[dr]^{\Yo^{\strg}}\ar[dl]_{\Yo}&\\ \CM\ar@{..>}[rr]^{\cp}&&\CM^{\strg}}\ .$$
To a  bornological coarse space $X$  with a coarse entourage $U$, we associate the Rips complex $P_{U}(X)$ in $\UBC$, defined as the realization of the simplicial complex whose $n$-simplices are the  $(n+1)$-tuples  $\{x_{0}
,\dots x_{n}\}$  of pairwise distinct  points in $X$ with $(x_{i},x_{j})\in U$ for all pairs $i,j$. 
A point in $P_{U}(X)$ will be considered as a finitely supported probability measure on $X$. We equip $P_{U}(X)$ with the spherical path metric, allowing infinite distances between points in different components.
We consider $P_{U}(X)$ as a uniform bornological coarse space with the metric uniform and coarse structures, and the bornology generated by the subsets $P_{U}(B)$ for all bounded subsets $B$ of $X$.
 By  
 \cite[Sec. 5]{ass} the Rips complex construction provides a functor  \begin{equation}\label{erfwerfrefwf}\BC\ni X\mapsto \colim_{U\in \cC_{X}}\Yo\cB(P_{U}(X))\in \Sp\cB
\end{equation} which turns out to be a coarse homology theory,  and  can therefore equivalently be interpreted as a colimit-preserving functor
$$\bP:\CM\to \Sp\cB\ .$$ Furthermore, the  local homology theory \eqref{sfdpokopsdvdsfvsdfvsdf} can be interpreted as a colimit-preserving functor 
$$\cO^{\infty,\strg}:\Sp\cB \to \CM^{\strg}\ .$$
The compositions $$\cF:\UBC\stackrel{c}{\to}\BC  \stackrel{\Yo}{\to} \CM, \quad \cF^{\strg}:\UBC\stackrel{c}{\to}\BC  \stackrel{\Yo^{\strg}}{\to} \CM^{\strg}$$
are alse local homology theories,  and induce colimit-preserving functors  $$\cF:\Sp\cB\to \CM\ , \quad \cF^{\strg}:\Sp\cB\to \CM^{\strg}\ .$$
Composing the cone boundary from \eqref{grewfwerfwef} with $\cp$ we get the natural transformation $$\partial^{\cone,\strg}:\Sigma^{-1} \cO^{\infty,\strg}\to \cF^{\strg}$$
of colimit-preserving functors from $\Sp\cB$ to $\CM^{\strg}$.

As explained in \cite[Sec. 5]{ass}, the natural maps  $X_{U}\to c(P_{U}(X))$ of bornological coarse spaces given by the inclusion of  $X$ as the zero-skeleton of the Rips complex induce an equivalence
$$d:\id_{\CM}\stackrel{\simeq}{\to} \cF \bP$$
 of endofunctors of $\CM$. Composing with  $\cp$ we get an equivalence 
 $$d^{\strg}:\cp\stackrel{\simeq}{\to} \cF^{\str}\bP$$ of functors from $\CM$ to $\CM^{\strg}$.
We can now state the definition of the motivic coarse assembly map.
\begin{ddd}[{\cite[Def. 9.7]{ass}}] \label{okopgtrgegrtg}The  motivic coarse assembly map is the natural transformation
\begin{equation}\label{euhwreiogevdfs}\mu: \Sigma^{-1}\cO^{\infty,\strg}\bP\stackrel{\partial^{\cone,\strg}}{\to}  \cF^{\strg}\bP\stackrel{d^{\strg,-1}}{\to} \cp
\end{equation}
of functors from $\CM$ to $\CM^{\strg}$.
\end{ddd}

\begin{rem}
Note that the  topological coarse $K$-homology \begin{equation} 
K\cX:\BC\to \Sp
\end{equation} constructed in \cite{buen}, \cite{coarsek} is strong and can therefore be interpreted as a colimit-preserving functor $K\cX:\CM^{\strg}\to \Sp$. 
Applying $K\cX$ to  \eqref{euhwreiogevdfs} and specializing to the motive $X$ in $\CM$
 yields the coarse assembly map    
 \begin{equation}\label{gewrferfwerfrefwref}
\mu_{K\cX,X}:\Sigma^{-1}K\cX(\cO^{\infty,\strg}\bP(X))\to K\cX(X)
\end{equation} for coarse $K$-homology and $X$.
 If $X$ is represented by a bornological coarse space of bounded geometry, then  \eqref{gewrferfwerfrefwref} is a 
  version  of  the classical coarse Baum-Connes assembly map. To this end, in  \cite{ass}, \cite{Bunke:2024ab} we have  identified the domains and   targets on the level of homotopy groups, but there is still no reference for the equality of the maps.
    \hB \end{rem}

In $\BC$ we have the subcategory of coarsely discrete objects and the larger subcategory of bounded geometry objects.  Their images under $\Yo$ or $\Yo^{\strg}$ generate  the localizing subcategories 
$$\CM_{\disc}\subseteq  \CM_{\bgeom} \ , \qquad 
\CM^{\strg}_{\disc}\subseteq  \CM^{\strg}_{\bgeom}\ ,$$ respectively.
We finally define the  localizing subcategory  \begin{equation}\label{bfbsdfbwe}\CM_{\cass}\subseteq \CM
\end{equation}   consisting of objects on which $\mu$ in \eqref{euhwreiogevdfs} is an equivalence.

Since $\mu$ is  known to be an equivalence on $\Yo(X)$ for discrete $X$ (see \cite[Prop. 10.1]{ass}) we have the inclusion  
$$\CM_{\disc}\subseteq \CM_{\cass}\ .$$

On the other hand, if $X$ in $\BC$ has bounded geometry, then  it follows from 
\eqref{erfwerfrefwf} that $$\cO^{\infty,\strg}\bP(X)\in \CM^{\strg}_{\disc}\ .$$ 
Here we use that $\cO^{\infty,\strg}$ is a local homology theory, and therefore sends finite-dimensional simplicial complexes with the spherical path metric to objects in $\CM^{\strg}_{\disc}$;  see \cite[Sec. 11]{ass} for  
the details of the induction argument. We further use bounded geometry and 
\cref{wtrhwrtgwtgtwg} in order to see 
  that   for every entourage $U$ in $\cC_{X}$, here exists $U'$ in $\cC_{X}$ such that $U\subseteq U'$ and the map
$P_{U}(X)\to P_{U'}(X)$ factorizes over a    finite-dimensional simplicial complex.

Thus, $$\cO^{\infty,\strg}\bP (\CM_{\bgeom})\subseteq \CM^{\strg}_{\disc}$$
and therefore
\begin{equation}\label{regeopwer}\cp(\CM_{\cass}\cap \CM_{\bgeom})\subseteq \CM^{\strg}_{\disc}\ .
\end{equation}

\begin{rem}The basic open problem is whether
 \begin{equation} \xymatrix{\CM_{\disc}\cap \CM_{\bgeom}\ar[r]^{\cp}\ar[d] &\CM^{\strg}_{\disc}\cap \CM^{\strg}_{\bgeom} \ar[d] \\ \CM_{\bgeom}\ar[r]^{\cp} &\CM^{\strg}_{\bgeom} } 
\end{equation} 
is a pullback; i.e., whether for a bounded geometry object $X$ in $\CM$ the condition
that $\cp(X)$ is discrete implies that $X$ is discrete.
In this case $$\CM_{\disc} \cap \CM_{\bgeom}=  \CM_{\cass}\cap \CM_{\bgeom} \ .$$ 
If the coarse homology theory $\cO^{\infty,\strg}\bP$ were strong (which is also not known), then the answer would be yes.
In this case the restriction of the coarse assembly map to bounded geometry objects
could be identified with the  counit of the adjunction (which exists unconditionally for formal reasons)
$$\incl:\CM^{\strg}_{\disc}\cap  \CM^{\strg}_{\bgeom} \leftrightarrows   \CM^{\strg}_{\bgeom}\ :R\ .$$
On the other hand, if this  were true, then the right adjoint $R$ itself would be equivalent to
$\Sigma^{-1} \cO^{\infty,\strg}\bP_{|\CM^{\strg}_{\bgeom}}$  and would therefore preserve colimits; i.e., would also be  a left adjoint. Because of this, it seems to be 
 unreasonable to believe that $\cO^{\infty,\strg}\bP$ is strong.
\hB
 \end{rem}
 
 
 \section{Sequence spaces and the $L^{2}$-index theorem}\label{hkoeprttrgegertg1}

Let $Y$ be a bornological coarse space that decomposes as a coarsely disjoint union of subspaces $(Y_{n})_{n\in \nat}$. We consider the big family $Y_{-}:=(Y_{\le k})_{k\in \nat}$
with $Y_{\le k}:= \bigsqcup_{n\le k} Y_{n} $.   
We furthermore consider a branched coarse $G$-covering   $f:X\to Y$  with respect to $Y_{-}$ (see \cref{jijhitrojhortrgrtgertgeg}) and set
 $X_{n}:=f^{-1}(Y_{n})$.   We define the big family $X_{-}:=f^{-1}(Y_{-})$ on $X$.
Using excision we obtain canonical maps
$$K\cX(Y )\to \prod_{n\in \nat} K\cX(Y_{n})\ , \quad K\cX^{G}(X )\to \prod_{n\in \nat} K\cX^{G}(X_{n})$$
which associate to a class $p$   its sequence of components $(p_{n})_{n}$.
We get induced maps 
$$K\cX(Y ,Y_{-})\to \frac{\prod_{n\in \nat} K\cX(Y_{n})}{ \bigoplus_{n\in \nat} K\cX(Y_{n})}\ , \quad K\cX^{G}(X,X_{-} )\to \frac{\prod_{n\in \nat} K\cX^{G}(X_{n})}{\bigoplus_{n\in \nat} K\cX^{G}(X_{n})}\ ,$$
which associate to a class $[p]$   the tail $[(p_{n})_{n\in \nat}]$ of the sequence of components. Note that in this case  only the tail is well-defined, not the components $p_{n}$ separately.

We now assume  for every $n$ in $\nat$ that $Y_{n}$ is bornologically bounded  and that $X_{n}$ is bornologically $G$-bounded.
We then  
  have  traces $\tau:\pi_{0} K\cX(Y_{n})\to \Z$ and  $G$-traces $\tau^{G}:\pi_{0} K\cX^{G}(X_{n})\to \R$, see 
 \cite[Sec. 9]{Bunke:2025aa} for details.
 Applying these traces to the tails of the sequences of components we get traces
$$\bar \tau:\pi_{0}K\cX (Y,Y_{-})\to\frac{\prod_{\nat}\Z}{\bigoplus_{\nat}\Z}\ , \quad \bar \tau^{G}:\pi_{0} K\cX^{G} (X,X_{-})\to\frac{\prod_{\nat}\R}{\bigoplus_{\nat}\R}\ .$$

We now assume in addition  that $X$ has finite asymptotic dimension. Then the transfer
$$f^{*}:K\cX(Y,Y_{-})\to K\cX^{G}(X,X_{-})$$ is defined by \cite[Cor. 8.8]{Bunke:2025aa}. Specializing the coarse assembly map  \eqref{gewrferfwerfrefwref} to the motive $\Yo(Y,Y_{-})$ in $\CM$ we 
 get  the morphism of spectra
\begin{equation}\label{gwrefwerfwerfwref}
\mu:\Sigma^{-1}K\cX(\cO^{\infty,\strg}\bP(Y,Y_{-}))\to  K\cX(Y,Y_{-})\ .
\end{equation} We consider a class $[p]$ in $\pi_{0}K\cX (Y,Y_{-})$.

Note that Assumption \ref{retgtrtbvgfdbgfd}.\ref{koperhtrtgetrgertg}   together with 
  Assumption \ref{retgtrtbvgfdbgfd}.\ref{ijfioerfjowewfqwef}  
  implies that   the terms in \eqref{grewfwerfrw34w} below are well-defined as explained above.  
The following  \cref{retgtrtbvgfdbgfd} is a version of 
 \cite[Lem. 6.5]{Willett_2012} in the context of the coarse assembly map $K\cX(\mu)$ with $\mu$ as in \eqref{okopgtrgegrtg}. 
 \begin{theorem}\label{retgtrtbvgfdbgfd}   
 We assume:  \begin{enumerate} 
 \item \label{koperhtrtgetrgertg} 
 The components $Y_{n}$ are bornologically bounded, and  $X_{n}$ are bornologically $G$-bounded for all $n$.
  \item \label{hwrtwrgwrgrtgt} $Y_{n}$ is a finite union of  coarsely bounded coarse components for every $n$ in $\nat$.  
  \item \label{ijfioerfjowewfqwef}  The bornological coarse space $X$ has finite asymptotic dimension.
    \item \label{gkpwerferfwfref} For a cofinal subset of entourages $U$ of $X$ we have:
 \begin{enumerate}   \item  \label{gergewferfwerferw} The bornological coarse space  $Y_{f(U)}$ has bounded geometry. 
  \item \label{kpogwregrefwrf} 
  The bornological coarse space $X_{U}$ has finite asymptotic dimension. 
  \end{enumerate}
 \end{enumerate}
 If    
 $[p]$ belongs to the image of the coarse assembly map $$K\cX(\mu ):\pi_{1}K\cX(\cO^{\infty,\strg}\bP(Y,Y_{-}))\to \pi_{0}K\cX(Y,Y_{-})\ ,$$ then we have
\begin{equation}\label{grewfwerfrw34w}
\bar \tau([p])=\bar \tau^{G}(f^{*}[p])\ .
\end{equation} 
\end{theorem}
\begin{proof}

By assumption, there exists $[v]$ in $\pi_{1}K\cX(\cO^{\infty,\strg}\bP(Y,Y_{-}))$ such that $\mu([v])=[p]$.  For the argument, we must unfold the definition of $\mu$ in this relation based on \cref{okopgtrgegrtg}.

  Unfolding the definition of $\bP$ and of  the evaluation of a coarse homology theory on the pair $(Y,Y_{-})$
  we can  choose an entourage $V$ of $Y$ and $n_{0}$ in $\nat$ 
such that 
$[v]$ comes from a class $v'$ in   $$\pi_{1} K\cX (\cO^{\infty}(P_{V}( Y)),\cO^{\infty}(P_{V }(Y_{\le n_{0}})))\ .$$
The class $v'$ has well-defined components
$v'_{n}$ in $ \pi_{1}K\cX (\cO^{\infty}(P_{V}(Y_{n})))$ for all $n$ in $\nat$ with $n> n_{0}$.
We let $d:Y_{V}\to P_{V}(Y)$ in $\BC$  be the coarse equivalence induced by the inclusion of the zero skeleton, and let
$d_{n}:Y_{V,n}\to P_{V}(Y_{n})$ be the restriction of this map to $Y_{n}$. After enlarging $V$ 
 we can furthermore assume that $[p]$ comes from a class $[p']$ in $\pi_{0}K\cX(Y_{V},Y_{-})$ and  that   $ \partial^{\cone}[ v']=d([p'])$ in $\pi_{0}K\cX(P_{V}(Y),P_{V}(Y_{-}))$, where $[ v']$ denotes the image of $v'$ in $\pi_{1} K\cX (\cO^{\infty}(P_{V}( Y)),\cO^{\infty}(P_{V }(Y_{-})))$. 
 This implies the equality 
 of tails $[(\partial^{\cone} v'_{n})_{n}]=[(d_{n}(p'_{n}))_{n}]$ in  $$\frac{\prod_{n\in \nat} K\cX_{0}(P_{V}(Y_{n}))}{ \bigoplus_{n\in \nat} K\cX_{0}(P_{V}(Y_{n}))}
   \ .$$

After increasing $n_{0}$ further,  we can assume, by \cite[Prop. 3.13]{Bunke:2025aa}, that the map
\begin{equation}\label{t34rw3fw}
P(f)_{> n_{0}}:P_{U }( X_{>n_{0}})\to P_{V }( Y_{>n_{0}})
\end{equation}
is a uniform $G$-covering, where $U  :=f^{-1}(V_{Y_{>n_{0}}} ) \cap P$  is the lift of $V_{>n_{0}}:=V\cap (Y_{>n_{0}}\times Y_{>n_{0}}) $ to a coarse entourage of $X$ such that $f(U)=V_{Y_{>n_{0}}}$.  Here    $P$ is the connection  witnessing the coarse covering (see \cref{jijhitrojhortrgrtgertgeg}).

We want to employ the cone transfer   \cite[Def. 5.6]{Bunke:2025aa}
$$\cO^{\infty}(P(f)_{>n_{0}})^{*}: K\cX(\cO^{\infty}(P_{V}(Y_{>n_{0}})))\to  K\cX^{G}( \cO^{\infty}(P_{U}(X_{>n_{0}}))) \ .$$
 To this end, we must assume that
 $  \cO^{\infty}(P_{U}(X_{>n_{0}}))$ has finite asymptotic dimension.
 
 The inclusion of the zero skeleton induces a coarse equivalence $(X_{> n_{0}})_{U}\to P_{U}(X_{>n_{0}})$.
  By Assumption \ref{retgtrtbvgfdbgfd}.\ref{gkpwerferfwfref}    we can assume, after enlarging $V$ and $n_{0}$ correspondingly so that \eqref{t34rw3fw} is still a uniform $G$-covering, that  $(X_{>n_{0}})_{U}$ has finite asymptotic dimension. In this case,
$P_{U}(X_{>n_{0}})$ also has finite asymptotic dimension. 

\begin{rem} If $P_{V}(Y_{>n_{0}})$ and hence $P_{U}(X_{>n_{0}})$ were finite-dimensional, then
$   \cO^{\infty}(P_{U}(X_{>n_{0}}))$ has finite asymptotic dimension by \cite[Cor. 4.10]{Bunke:2025aa}.
 This is the case if $Y$ has strongly bounded geometry  in the sense that for every entourage $V$, we
have $\sup_{y\in Y} |V[y]|<\infty$.  
This happens e.g.  in the case of 
 Higson's counter example to the coarse Baum-Connes conjecture; see \cite[Sec. 12]{Bunke:2025aa} for details. But in the present paper     we only assume bounded geometry and therefore need a workaround.
 \hB
\end{rem}

 \begin{construction}\label{wtrhwrtgwtgtwg}{\em 
 In this construction, $Y$ can be any bornological coarse space with a symmetric entourage   $V$.
We can choose a  maximal $V$-discrete subset $Z$ of $Y$. Then we have   $\bigcup_{z\in Z} V^{2}[z]=Y$.
We then construct a triangle that  commutes up to a uniform homotopy:
$$\xymatrix{P_{V}(Y)\ar[dr]^{\kappa}\ar[rr]^{\can}&&P_{V^{5}}(Y)\\&P_{V^{5}}(Z)\ar[ur]^{\incl}&}\ .$$
In order to define the map $\kappa$ we first choose, using the $V^{2}$-density of $Z$ in $Y$,  a map $\kappa_{0}:Y\to Z$  such that $(\kappa_{0}(y),y)\in V^{2}$
for all $y$ in $Y$. The map $\kappa_{0}$ linearly  extends to a map $\kappa:P_{V}(Y)\to P_{V^{5}}(Z)$.
We now define the affine  homotopy
$$u\mapsto u \ \can +(1-u)\ \incl\circ \kappa:P_{V}(Y)\to P_{V^{5}}(Y)\ ,$$ where we interpret points in the Rips complexes as finitely supported probability measures.   }\hB \end{construction}
Back to the current proof, by Assumption  \ref{retgtrtbvgfdbgfd}.\ref{gkpwerferfwfref}, we can increase $U$ 
such that $X_{U}$ has finite asymptotic dimension and  $Y_{V}$ has bounded geometry at the same time, where  $V:=f(U)$.  By   \cref{wtrhwrtgwtgtwg} we can assume, after  replacing $V$ by $V^{5}$ and enlarging $n_{0}$   so that \eqref{t34rw3fw} is still a uniform $G$-covering,  that  there exists  
 a subset $Z  $ of $ Y_{>n_{0}} $ such that
 $P_{V }(Z )$ is finite-dimensional and 
   $v'$  is the image   under the canonical inclusion $P_{V}(Z )\to P_{V}(Y_{>n_{0}})$  of  a class $v''$ in $\pi_{1}K\cX(\cO^{\infty}(P_{V}(Z )))$.
 In addition, the sets $Z_{n}:=Z\cap Y_{n}$ are finite for all $n$ in $\nat$ by Assumption \ref{retgtrtbvgfdbgfd}.\ref{hwrtwrgwrgrtgt}.    Similarly, the class $d([p'])$ comes from a class $[p'']$ in $\pi_{0}K\cX( P_{V}(Z ) ,  P_{V}(Y_{-}\cap Z )) $
 and we have $[\partial^{\cone}v'']=[p'']$. 
 
 We set $\hat Z:=f^{-1}(Z) $.
 The maps $f_{V}:(X_{U}\to Y_{V},Y_{-})$, 
 $(\hat P(f):P_{U }(\hat Z)\to P_{V }(Z),P_{V}(Z_{-}))$ and $(P(f):P_{U}(X)\to P_{V}(Y),P_{V}(Y_{-}))$ are all branched  coarse $G$-coverings. For $f_{V}$ this follows from \cite[Lem. 2.17]{Bunke:2025aa} applied to the canonical map $Y_{V}\to Y$ saying that the pullback of a branched coarse $G$-covering is again a branched coarse $G$-covering. Since $d$  and the induced  
$G$-equivariant map  $X_{U}\to P_{U}(X)$ are  coarse equivalences  we can  conclude
  that  $P(f)$  is a branched coarse $G$-covering. We finally use    \cite[Lem. 2.17]{Bunke:2025aa}
  again to conclude that 
   $\hat P(f)$ a branched coarse $G$-covering.
 
 By the naturality of transfers (which is part of  \cite[Def. 5.1]{Bunke:2025aa}), the diagram 
 $$\xymatrix{K\cX( P_{V}(Z) ,  P_{V}(Z\cap Y_{-}))\ar[r]\ar[d]^{\hat P(f)^{*}} &K\cX( P_{V}(Y) ,  P_{V}(Y_{-})\ar[d]^{P(f)^{*}}&\ar[l]_-{d} \ar[d]^{f_{V}^{*}} K\cX(Y_{V},Y_{-})\ar[r]&K\cX(Y,Y_{-})\ar[d]^{f^{*}}\\ 
  K\cX^{G}( P_{U}(\hat Z) ,  P_{U}(\hat Z\cap X_{-}))\ar[r]&K\cX^{G}( P_{U}(X) ,  P_{U}( X_{-}))&\ar[l]  K\cX^{G}(X_{U},X_{-})\ar[r]&K\cX^{G}(X,X_{-})}$$
 commutes. The maps in the upper line identify the classes
 $[p'']$, $d([p'])$, $[p']$ and $[p]$.
 By the naturality of the traces, it suffices to show that
  \begin{equation}\label{qfwedqwdwqedqewdq}
\bar \tau ([p''])=\bar \tau^{G}(\hat P(f)^{*} [p''])\ .
\end{equation} 

The map  
$\hat P(f) :P_{U }(\hat Z )\to P_{V }(Z )$ is a uniform $G$-covering.
Since $P_{U }(\hat Z)\subseteq P_{U }(  X_{>n_{0}})$
 we conclude that $P_{U }(\hat Z)$ has finite asymptotic dimension. Since $P_{U }(\hat Z )$ is a finite-dimensional simplicial complex
we see by \cite[Cor. 4.10]{Bunke:2025aa} that
$\cO^{\infty}(P_{U}(\widehat Z ))$ has finite asymptotic dimension.
The cone transfer
\begin{equation}\label{bfgdbdgbergbf}\cO^{\infty}(\hat P(f) )^{*}:  K\cX(\cO^{\infty}(P_{V }(Z )))\to  K\cX^{G}( \cO^{\infty}(P_{U }(\hat Z )))
\end{equation} 
  is now defined by \cite[Def. 5.6]{Bunke:2025aa}. 
We argue with the following diagram, which commutes by the naturality of transfers and cone boundaries, and where   $n>n_{0}$:
$${\tiny \hspace{-2.5cm}\xymatrix{ \ar[d]\Sigma^{-1}K\cX(\cO^{\infty}(P_{V}( Z_{n})))\ar@/_1cm/[ddddd]_{ \cO^{\infty}(\hat P(f)_{n})^{*}}\ar[rrr]^{\partial^{\cone}}&&&K\cX(P_{V}( Z_{n}))\ar[d]
\\ \Sigma^{-1}K\cX(\cO^{\infty}(P_{V }( Z )))\ar[dr]_{\simeq} \ar[ddd]^{\cO^{\infty}(\hat P(f) )^{*}}\ar[rrr]^-{\partial^{\cone}} &&& K\cX(P_{V }(Z ))\ar[dl]  \\ 
& \Sigma^{-1}K\cX(\cO^{\infty}(P_{V }(Z)) , \cO^{\infty}(P_{V }(Z \cap Y_{\le n_{0}}))) \ar[r]^-{[-]\circ \partial^{\cone}}\ar[d]^{ \cO^{\infty}(P(f))^{*}}&  K\cX(P_{V }(Z) ,P_{V }(Z \cap Y_{-})) \ar[d]^{\hat P(f)^{*}}&\\&\Sigma^{-1}K\cX^{G}(\cO^{\infty}(P_{U }(\hat  Z)), \cO^{\infty}(P_{U }(\hat Z\cap X_{\le n_{0}} )))\ar[r]^-{[-]\circ \partial^{\cone}}& K\cX^{G}(P_{U }(\hat Z), P_{U }(\hat Z\cap X_{-} ))&\\ \Sigma^{-1}K\cX^{G}(\cO^{\infty}(P_{U }(\hat Z )))\ar[ur]^{\simeq} \ar[rrr]^-{\partial^{\cone}} &&& K\cX^{G}(P_{U }( \hat Z ) )\ar[ul]  \\ \Sigma^{-1}K\cX^{G}(\cO^{\infty}(P_{U}(\hat Z_{n})))\ar[rrr]^{\partial^{\cone}}\ar[u]&&& K\cX^{G}(P_{U}(\hat Z_{n}))\ar[u]} }\ .$$
It shows that
the tail of the sequence of  components of  $\hat P(f)^{*} [\partial^{\cone} v'']$
is given by the tail of the family $(\partial^{\cone} \cO^{\infty}(\hat P(f)_{n})^{*} v''_{n})_{n>n_{0}}$.
 Since all the components $Z_{n}$ are finite, the $P_{V}(Z_{n})$ are finite simplicial  complexes and  we can apply 
  the  $L^{2}$-index theorem \cite[Thm. 11.3]{Bunke:2025aa} stating  that 
\begin{equation}\label{werfwerfwrefwerfref}
\tau^{G} (\partial^{\cone} \cO^{\infty}(\hat P(f)_{n})^{*} v''_{n})=\tau (\partial^{\cone} v_{n}'')
\end{equation}
for every $n$ in $\nat$ with $n>n_{0}$.
The  desired equation \eqref{qfwedqwdwqedqewdq}  now follows from   
  \begin{eqnarray*}
\bar\tau^{G}(\hat P(f)^{*}[p''])&=&  \bar \tau^{G} (\hat P(f)^{*}[ \partial^{\cone} v''])\\&=& \bar \tau^{G} ([ \partial^{\cone} \cO^{\infty}(\hat P(f)_{>n_{0}})^{*}v''])\\&=&[ (\tau^{G} (\partial^{\cone} \cO^{\infty}(\hat P(f)_{n})^{*} v''_{n}))_{n}]\\&\stackrel{\eqref{werfwerfwrefwerfref}}{=}&
[ (\tau (\partial^{\cone}   v''_{n}))_{n}]\\&=&\bar \tau( [\partial^{\cone} v''])\\&=&\bar \tau( [p''])\ .
\end{eqnarray*}
  \end{proof}

  \begin{rem}
  The Assumption  \ref{retgtrtbvgfdbgfd}.\ref{koperhtrtgetrgertg} and  \ref{retgtrtbvgfdbgfd}.\ref{ijfioerfjowewfqwef}   are necessary in order to define the terms in the statement \eqref{grewfwerfrw34w} of the theorem.
         The  Assumptions  \ref{retgtrtbvgfdbgfd}.\ref{gkpwerferfwfref}.\ref{gergewferfwerferw}-\ref{kpogwregrefwrf} ensure that the transfers
  in the domain of the coarse assembly map, e.g. the map \eqref{bfgdbdgbergbf},  are defined.
    Note that none of 
  the Assumptions \ref{retgtrtbvgfdbgfd}.\ref{gkpwerferfwfref}.\ref{kpogwregrefwrf}
  and  \ref{retgtrtbvgfdbgfd}.\ref{gkpwerferfwfref}.\ref{ijfioerfjowewfqwef} implies the other. 
  If we are interested in the statement for a single class $[p]$, then we could replace
  $X$ by $X_{U}$ and $Y$ by $Y_{f(U)}$ for a sufficiently large $U$
  satisfying  \ref{retgtrtbvgfdbgfd}.\ref{gkpwerferfwfref}.\ref{kpogwregrefwrf} so that $[p]$ comes from
  $Y_{f(U)}$.  Then $f^{*}[p]$ is  defined (but not the original map $f^{*}$), 
 and the statement still holds.

  Finally, Assumption  \ref{retgtrtbvgfdbgfd}.\ref{hwrtwrgwrgrtgt} ensures that the $L^{2}$-index theorem can be applied to the components.  The finite asymptotic dimension assumptions could be replaced by more general assumptions of an operator norm localization property, but then we cannot cite  \cite{Bunke:2025aa} anymore.
\hB  \end{rem}
  
    \begin{rem}\label{hkoperhtetrgetgrt}

The statement   \eqref{grewfwerfrw34w} is the same as in   \cite[Lem. 6.5]{Willett_2012}. But formally  it does not  follow from this reference since the terms have different technical definitions.    \cite[Lem. 6.5]{Willett_2012} is stated in the case where $Y$ is a  space of graphs in the sense of   \cite[Def. 1.1]{Willett_2012}, but it is clearly true  and also used in the literature
in more general situations. With  \cref{retgtrtbvgfdbgfd}.\ref{koperhtrtgetrgertg} 
we hope to provide a reference  for such a  result  with explicit, very general assumptions (not excluding  the possibility to generalize further by replacing  finite asymptotic dimension by the operator norm localization property).
\hB
 \end{rem}

 \section{Proof of \cref{kopherthretge9}}\label{hkoeprttrgegertg}

We  consider $X$ in $\Fun(BG,\UBC)$ satisfying the assumptions of \cref{kopherthretge9}. In this section, we show,
 in analogy to   \cite[Sec. (4.2)]{Kitsios:2025aa},  that  $$\Yo(\cO^{\infty}(X)//G)\not\in \CM_{\cass}\ .$$
 More concretely, we will show that the specialization of the  coarse assembly map  (see  \eqref{gewrferfwerfrefwref})
 \begin{equation}\label{cgtrtegrtg}
\mu_{K\cX,\cO^{\infty}(X)//G}: 
 \pi_{2}K\cX(\cO^{\infty,\strg}\bP(\cO^{\infty}(X)//G))\to \pi_{1}K\cX(\cO^{\infty}(X)//G)
\end{equation}
 is not surjective.
 
 As will be explained in \cref{kophjkertophrtgertgrteg9},    under the assumptions of \cref{kopherthretge9}, we have  a $\Z$-invariant class $p$ (defined by  \eqref{whtgwergergws}) in $\pi_{0}K\cX(\Sq(X)//G )$ whose components
 $p_{n}$ in $\pi_{0}K\cX(\Sq_{n}(X)//G) $ are represented by one-dimensional projections.
 
  We let $[p]$ denote the image of $p$ in the relative coarse $K$-homology group $$\pi_{0}K\cX(\Sq(X)//G,\Sq_{-}(X)//G)\ .$$ 
Since $[p]$ is $\Z$-invariant, the class 
 $(\pm 1)[p]$ (see \cref{kopggwregweg} for notation)  is annihilated by $1+t$ in the fibre sequence  \eqref{rewgerfewrfwe}. Using this fibre sequence and copying the ideas from \cite[Sec. (4.2)]{Kitsios:2025aa},
  there  exists a class $u$ in $\pi_{1}K\cX(\cO^{\infty}(X)//G)$ such that
$\delta (u)=(\pm 1) [p]$.   In order to show \cref{kopherthretge9}, it suffices   to show that $u$ does not belong to the image of the assembly map \eqref{cgtrtegrtg}.  
 
Assume for a contradiction   that $$u=\mu_{K\cX, \cO^{\infty}(X)//G}(w)$$ for some class $w$ in $\pi_{2}K\cX(\cO^{\infty,\strg}\bP(\cO^{\infty}(X)//G))$.
Since the coarse assembly map is a natural transformation of coarse homology theories, it is compatible with Mayer-Vietoris boundary maps. In particular, we have an  equivalence
$$\delta \circ \mu_{K\cX, \cO^{\infty}(X)//G}\simeq \mu_{K\cX,\Yo(\Sq(X)//G,\Sq_{-}(X)//G)} \circ \delta\ .$$
  We can  then conclude that  \begin{equation}\label{hrtegertgetgetgergtr}
 (\pm 1)[p]= \mu_{K\cX,\Yo(\Sq(X)//G,\Sq_{-}(X)//G)} (v)
\end{equation}
  for   $v:=\delta(w)$ in $$\pi_{1} K\cX(\cO^{\infty,\strg}\bP(\Sq(X)//G,\Sq_{-}(X)//G))\ .$$

 In view of the construction of $p$ in  \cref{kophjkertophrtgertgrteg9} and by \cref{jophrthetgetgetg},
we can assume that   $p$ comes from a class $p'$ in $\pi_{0}K\cX(\Sq(X)_{V}//G) $,
where  $V$ is  some symmetric $G$-adapted 
 entourage of $\Sq(X)$. 
The class $p'$ has well-defined components $p'_{n}$ in $\pi_{0}K\cX( \Sq_{n}(X)_{V_{n}}//G)$. \
  The sequence of traces of the components is given by $(\tau(p'_{n}))_{n }=(1)_{n}$ in $ \prod_{\Z}\Z$.   But note that the classes $p$ and $p'$ are more than the sequences of their components; see the following \cref{okwpogwkrpog}.

\begin{ex} \label{okwpogwkrpog}Here is an example of a non-trivial class in $\pi_{0}K\cX(\Sq(X)//G)$  (even $\Z$-invariant)  with zero components, see \cite[(4.2)]{Kitsios:2025aa}.
For simplicity, we assume that every uniform entourage of $X$  generates the maximal coarse structure of $X$. This is the case if $X$ is a path-metric space of bounded diameter. By this additional assumption, we have 
$K\cX_{0}(\Sq_{n}(X))\cong\pi_{0} KU\cong \Z$ for every $n$ on $\Z$,
and this isomorphism is given by the trace. The inclusion of a point $*\to X$  induces the first map in the composition \begin{equation}\label{htregetrgrtgert}
K\cX(*)\to K\cX(X_{\disc})\stackrel{!}{\to} K\cX(\Sq(X_{\disc}))\to  K\cX(\Sq(X_{\disc})//G)\to  K\cX(\Sq(X)//G)\ ,
\end{equation} where the marked map is a transfer as in  \cref{oijoegwerferfwref}.
   We let $e$ in $\pi_{0}K\cX(\Sq(X)//G)$ denote the image of $1$  in $\Z\cong \pi_{0}K\cX(*)$ under this composition. Since $e_{n}$ and $p_{n}$ both have trace $1$ we conclude that 
    $p_{n}=e_{n}$ for every $n$ in $\Z$.
Therefore $p-e$ has zero components. 
Since the maps  in \eqref{htregetrgrtgert} starting from $K\cX(\Sq(X_{\disc}))$ are induced by morphisms of bornological coarse spaces and the coarse assembly map is an equivalence for the discrete space $\Sq(X_{\disc})$ the class $e$ is in the image of the coarse assembly map.
We will show in the present section that $[p]$ is not in the image of the coarse assembly map. 
We can therefore conclude that $e-p\not=0$.
\hB
\end{ex}

By \cref{uoiwgrgwegwerfwfwerf},   the map
\begin{equation}\label{gregeferfrfw}
f:(G\ltimes \Sq(X)_{V})//G\to \Sq(X)_{V}//G
\end{equation}  is a branched coarse $G$-covering with respect to the big family  $ \Sq_{-}(X)_{V}$. 
We want to apply \cref{retgtrtbvgfdbgfd} to \eqref{gregeferfrfw}.
In the following, we verify the assumptions.

Since $X$ is bornologically bounded, we can conclude  that for all $n$ in $\nat$ the components  $\Sq_{n}(X)_{V}//G$ are bornologically  bounded and that the components 
$(G\ltimes \Sq(X)_{V})_{n}//G$ are bornologically  $G$-bounded.
Hence $\Sq(X)_{V}//G$ satisfies Assumption \ref{retgtrtbvgfdbgfd}.\ref{koperhtrtgetrgertg}.
 By Assumption \ref{kopherthretge9}.\ref{hpkprhtrgrtegerg1}, 
every uniform entourage of $X$ generates a coarse structure in which $X$
  has finitely many coarsely bounded coarse components.
 Therefore,
the bornological coarse space $\Sq(X)_{V}//G$ satisfies
Assumption \ref{retgtrtbvgfdbgfd}.\ref{hwrtwrgwrgrtgt}.

Recall the notion of   finite uniform topological dimension \cite[Def. 4.5]{Bunke:2025aa}.
If $X$ has finite  uniform topological dimension   and finite (coarse) asymptotic dimension, then $\Sq(X)$ 
has finite asymptotic dimension by \cite[Cor. 4.9]{Bunke:2025aa}. In our present situation, we need that 
$\Sq(X)_{V}$ has finite asymptotic dimension for every $G$-adapted entourage $V$. The coverings with Lebesgue entourage $V$ provided  by  \cite[Cor. 4.9]{Bunke:2025aa} are bounded in the coarse structure of $\Sq(X)$, but not necessarily bounded in the smaller coarse structure of $\Sq(X)_{V}$.  
 In order to ensure the latter condition, we need the following stronger finiteness condition.

 Let $X$ be in $\UBC$.  
\begin{ddd}\label{asna} 
 A uniform scale $(V_{n})_{n\in \nat}$ has  finite Assouad-Nagata dimension 
 if there exist $d$ and $k$ in $\nat$ such that for every $l$ and $n$ in $\nat$ there exists
 a $V_{n}^{kl}$-bounded covering of multiplicity $d$ with Lebesgue entourage $V^{l}_{n}$.
    \end{ddd}
    
    \begin{rem}
    We use the term {\em Assouad-Nagata dimension}
   since we consider \cref{asna} as the natural generalization of the classical notion of finite 
Assouad-Nagata dimension from metric spaces to uniform spaces.
If the uniform scale  is a metric scale  $(V_{r(n)})_{n\in \nat}$ with $\lim_{n\to \infty} r(n)=0$, then
the condition boils down to the classical notion.
 \hB
\end{rem}
\begin{rem}
One could try to define finite Assouad-Nagata dimension as the condition that there exists
$d$ and $k$ in $\nat$ such that for every uniform entourage $V$ of $X$ there exists
a $V^{k}$-bounded covering with multiplicity bounded by $d$ and Lebesgue entourage $V$.
But this condition might be too strong even in the metric case, since it must also be satisfied   for  entourages $V$ that are not contained
in any bounded metric entourage.  
 \hB
\end{rem}

For simplicity, we assume that $X$ has the maximal coarse structure, as required anyway by Assumption \ref{kopherthretge9}.\ref{lptgeertgetrg}. Recall that $V$ is a $G$-adapted symmetric  coarse entourage of $\Sq(X)$.
After enlarging $V$ further, we can  assume that $V_{n}=X\times X$ for $n<0$.
 
 \begin{lem}\label{koptwehrtherthrteeg}
 If  the uniform scale $(V_{n})_{n\in \nat}$   has finite   Assouad-Nagata dimension,   then 
  $\Sq(X)_{V}//G$ has bounded geometry. 
  \end{lem}
  \begin{proof}
    We will first show, under the assumption that $(V_{n})_{n\in \nat}$ is a uniform scale  with   finite Assouad-Nagata dimension, that $\Sq(X)_{V}$ has bounded geometry. Then we use the Lipschitz condition in order to  conclude that
  $\Sq(X)_{V}//G$ has bounded geometry too.

It suffices to show for any $V^{2}$-discrete subset $Z$ of $\Sq(X)_{V}$ that for every $l$ in $\nat$ we have  $\sup_{y\in \Sq(X)} |V^{l}[y]\cap Z|<\infty$.

If $y\in \Sq_{<0}(X)$, then $  |V^{l}[y]\cap Z| \le 1$.  It therefore  remains to consider $y$ in $\Sq_{\ge 0}(X)$.
      Let $k$ and $d$  be as in   \cref{asna}. 
   For every $n$ in $\nat$ we choose  a $V^{k}_{n}$-bounded covering of $X$ with Lebesgue entourage $V_{n}$ and multiplicity bounded by $d$.
   These coverings together provide a $V^{k}$-bounded covering $\cW=(W_{i})_{i\in  I}$  of $\Sq_{\ge 0}(X)$ with multiplicity bounded by $d$ and   with Lebesgue entourage $ V$.
Therefore, for every $m$ in $\nat$, we have   $|\{i\in I\mid Y\cap W_{i}\not=\emptyset\}|\le d^{2m}$  for every $V^{m}$-bounded subset $Y$ of $\Sq_{\ge 0}(X)$.  Note that $V^{l}[y]$ is $V^{2l}$-bounded for every $y$ in $\Sq_{\ge 0}(X)$.  
   By the $V^{2}$-discreteness of $Z$   for every $i$ in $I$ we have $|W_{i}\cap Z|\le 1$. We conclude that   
   \begin{equation}\label{wercfvsvdfv}\sup_{y\in \Sq_{\ge 0}(X)} |V^{l}[y]\cap Z|\le \sup_{y\in \Sq_{\ge 0}(X)} |\{i\in I\mid V^{l}[y]\cap W_{i}\not=\emptyset\}| \le d^{2l}<\infty \ .
\end{equation}
 This finishes the proof of bounded geometry for $\Sq(X)_{V}$.

Let $F$ be a finite subset of $G$ and set \begin{equation}\label{oerwigjregvweorgwergwr}W_{F}:=\{(gy,y)\mid f\in F, y\in \Sq(X)\}\ .
\end{equation} Since $V$ is $G$-adapted (see \cref{herjtogitrjgoetrgrtegetrg}), the uniform scale $(V_{n})_{n\in \nat}$ is a Lipschitz scale for the $G$-action. Therefore 
  the entourages $W_{F}\circ V^{l}$ for all finite subsets $F$ of $G$ and $l$ in $\nat$ are cofinal in the coarse structure of
$\Sq(X)_{V}//G$. We have $W_{F}[V^{l}[y]]=\bigcup_{g\in F} g V^{l}[y]$. Let $r$ in $\nat$ be a bound on the Lipschitz constants (see \cref{kohpertgetrgeg1})
of the elements $g$ in $F$.  Then  $W_{F}[V^{l}[z]]\subseteq \bigcup_{g\in F}  V^{rl}[gy]$.   It follows that
$ \sup_{y\in \Sq(X)}|W_{F}[V^{l}[y]]\cap Z|\le |F| \sup_{y\in \Sq(X)} |V^{rl}[y]\cap Z|<\infty$.
  \end{proof}

 For simplicity, we again assume that $X$ has the maximal coarse structure.  \begin{lem}\label{pltherthtergegrtg} Under the same assumptions as in \cref{koptwehrtherthrteeg}     the bornological coarse space
  $\Sq(X)_{V}$ has finite asymptotic dimension.
\end{lem}
\begin{proof} Let $d$ and $k$ be as in \cref{asna}.
We  fix $l$ in $\nat$. For every $n$ in $\Z$ we choose a $V_{n}^{kl}$-bounded covering of $X$ with multiplicity bounded by $d$  and with Lebesgue entourage $V_{n}^{l}$. For $n<0$, we take the one-member covering $(X)$.

Putting these coverings together, we 
get a $V^{kl}$-bounded covering of $\Sq(X)$ with multiplicity bounded by $d$ and Lebesgue entourage $V^{l}$.
\end{proof}

Recall our standing hypothesis that $G$ is finitely generated.
 The following is a version of the if-direction of \cite[Lem. 2.9]{dsa}; see also \cite[Prop. 7.2]{Sawicki_2018}. 

\begin{lem}\label{okgowpgrefwref}  Under the same assumptions as in \cref{koptwehrtherthrteeg} 
and if, in addition $G_{\can}$ has finite asymptotic dimension, then 
 $((G\ltimes \Sq(X)_{V})//G)_{U}$ has finite asymptotic dimension for a cofinal set of entourages $U$ of $(G\ltimes \Sq(X)_{V})//G$.
\end{lem}
\begin{proof}   
Assume that the  asymptotic dimension of $G_{\can}$ is bounded by $d_{G}$. By  \cref{pltherthtergegrtg}, we can bound the asymptotic dimension of $\Sq(X)_{V}$ by some number $d_{  X}$.

We fix a finite generating set $B$ of $G$ and define $U_{B}$ as in  \eqref{ertkhertgrtgge}.
The entourages of the form 
 $U(k,l):=(U_{B}^{k}\times \diag(\Z \times X))\circ  V^{l}_{G}$ for $k,l$ in $\nat$  are cofinal in the coarse structure of $(G\ltimes \Sq(X)_{V})//G$.  
 
 We fix $k,l$. Using   the  Lipschitz  condition, we can see that for $r$ in $\nat$ we have $U(k,l)^{r}\subseteq U(k',l')$ for suitable  $k',l'$ in $\nat$.
  By our assumptions, we can find  a coarsely bounded  covering $\cY$ of $G_{\can} $  with Lebesgue entourage $U_{B}^{k'}$ and
with multiplicity bounded by $d_{G}+1$.  Furthermore,
we can find a coarsely bounded covering $\cZ$ of $\Sq(X)_{V}$ with multiplicity bounded by $d_{X}+1$
and Lebesgue entourage $V^{l'}$. Then 
$$\pr_{G}^{-1}\cY\cap f^{-1}\cZ:=\{\pr_{G}^{-1}Y\cap  f^{-1}Z\mid Y\in \cY \:\& \:   Z\in \cZ   \}$$
is  a  covering of $G\times \Sq(X)$ with multiplicity bounded by $(d_{G}+1)(d_{ X}+1)$   
that  is coarsely bounded in the structure $((G\ltimes \Sq(X)_{V})//G)_{U(k,l)}$ and has Lebesgue entourage  $U(k,l)^{r}$, in fact even $U(k',l')$.

As we can choose $r$ in $\nat$ arbitrarily 
we can conclude that $((G\ltimes \Sq (X)_{V})//G)_{U_{k,l}}$ has finite asymptotic dimension. \end{proof}

By Assumption \ref{kopherthretge9}.\ref{rkhoeprtgertgretgegtr}, the assumptions of \cref{koptwehrtherthrteeg} and of \cref{okgowpgrefwref}
are satisfied for a cofinal set of $G$-adapted entourages $V$. 
We furthermore note that for the entourages $U$ of $(G\ltimes \Sq(X)_{V})//G$ produced by \cref{okgowpgrefwref}, the space $(\Sq(X)_{V}//G)_{f(U)}$ has bounded geometry
since the entourages  $f(U)$  generate the coarse structure of  $\Sq(X)_{V}//G$ and therefore \cref{koptwehrtherthrteeg} applies. Therefore, 
these lemmas verify the remaining  
Assumption \   \ref{retgtrtbvgfdbgfd}.\ref{gkpwerferfwfref}
 of \cref{retgtrtbvgfdbgfd}.
 By  \cref{retgtrtbvgfdbgfd}, we get $$0\not=[ (\pm 1)_{n}]=\tau( (\pm 1)[p'])=\bar \tau^{G}(f^{*}((\pm 1)[p']))\ .$$ 
 On the other hand,  
we use \cref{hpopethrtgertgtg} saying that $p$ is represented by a ghost projection $\hat P^{s}$. This implies by  \cite[Cor. 8.12]{Bunke:2025aa} that $f^{*}([p'])=0$ and hence also $f^{*}((\pm 1)[p'])=0$. This  is  the desired contradiction.
 
This finishes the proof of \cref{kopherthretge9} modulo the construction of the class $p$ which will be given in 
 \cref{kophjkertophrtgertgrteg9} below.

\begin{rem} 
The corresponding arguments  in the papers \cite{Kitsios:2025aa}, \cite{dsa}, \cite{Li_2023} use a different model of the coarse assembly map  and  refer to \cite[Lem. 6.5]{Willett_2012} 
for the step where the   $L^{2}$-index theorem is applied.
In our argument,  we propose to use  the model of the coarse assembly map
introduced in \cite{ass} and  \cref{retgtrtbvgfdbgfd} instead.     
 \hB
\end{rem}

 \section{Controlled Hilbert spaces}\label{kophjkertophrtgertgrteg9}

We consider $X$ in $\Fun(BG,\UBC)$ and assume that it satisfies the  Conditions \ref{lptgeertgetrg}-\ref{fhqwekfqweq} of    \cref{kopherthretge9}. 

 Following \cite{drno}, \cite{dsa}, \cite{Li_2023},
 we will explain in detail how to associate to  
 a suitably  ergodic $G$-invariant probability measure $\nu$ on $X$ a $\Z$-invariant class $p$ in $\pi_{0}K\cX(\Sq(X) //G)$ whose components are generators.
    This class is the main input for the proof of \cref{kopherthretge9} given in \cref{hkoeprttrgegertg}. 

Recall that the spectrum $K\cX(Y)$ for $Y$ in $\BC$ is defined in  \cite{buen}  as the topological $K$-theory spectrum
of the $C^{*}$-category $\bV(Y)$  of locally finite $Y$-controlled Hilbert spaces that
 are determined on points, and whose morphisms are bounded  operators  that can be approximated by controlled operators. It therefore requires some work to put the  class represented by the Drutu-Nowak projection $\hat P$ (described in \cref{fewfewrdefwerfwef})  into this framework.

The uniform structure on $X$ induces a topology and therefore a Borel measurable structure. Since $G$ acts by uniform maps, it also 
  acts by measurable maps.
 The additional datum going into the construction of $p$ is a $G$-invariant   Borel probability measure $\nu$.
   The group $G$ acts on the Hilbert space $H:=L^{2}(X,\nu)$ by isometries $g\mapsto \rho(g)$.
 We let $H\cong H^{G}\oplus (H^{G})^{\perp}$  be the decomposition of $H$ into the subspace of $G$-invariant vectors  and   its orthogonal complement.
 The following explains a part of the assumption   \cref{kopherthretge9}.\ref{kpbgbrgebgrbgrb}.
\begin{ddd}[{\cite[Def. 2.4]{drno}}]\label{gkwopergwerferfw}
We say that the action of $G$ on $H$ has a spectral gap if 
there exists a finite subset $S$ of $G$ and $\kappa$ in $(0,1)$ such that  for every $h$ in $(H^{G})^{\perp}$ we have 
$$\sup_{s\in S} \|h-\rho(s)h\|\ge  \kappa \|h\|\ .$$
\end{ddd}
If $G$ is finitely generated, then it suffices to check this condition on a finite generating set $S$ of $G$.
In this case we form the selfadjoint bounded operator 
$$M_{S}:=\frac{1}{|S|}\sum_{g\in G}  \rho(g)$$ on $H$. 
Note that $H^{G}$ is contained in the $1$-eigenspace of $M_{S}$.  
If the action of $G$ on $(X,\nu)$ is ergodic, then $H^{G}$ consists precisely of the constant functions.

\begin{prop}[{\cite{drno}}] \label{hehrtgertg}If $G$ acts ergodically on $(X,\nu)$, then the following are equivalent:
\begin{enumerate}
\item 
$1$ is an isolated
eigenvalue of $M_{S}$ with a one-dimensional eigenspace. 
\item The action of $G$ on $H$ has a spectral gap.
 \end{enumerate}
\end{prop}

 From now on, we assume that $X$  satisfies all conditions of \cref{kopherthretge9}.
 Using the counting measure on $\Z$ we equip  $ \Z\times X$ with the product measure $\hat \nu:=\delta\times \nu$  and define the Hilbert space $\hat H:=L^{2}(\Z\times X, \hat \nu)$.  
This Hilbert space again carries a unitary action $g\mapsto \hat \rho(g)$ of $G$ by  translations
on the $X$-factor and a unitary action $n\mapsto \hat t(n)$ of $\Z$ by translations on the  $\Z$-factor. 
 
 \begin{ddd}\label{fewfewrdefwerfwef} The   orthogonal projection $\hat P$ onto the subspace of $\hat H$ consisting of functions that are constant on the components $\{n\}\times X$ for all $n$ in $\Z$  is called the Drutu-Nowak projection. \end{ddd}  The Drutu-Nowak projection is obviously $\Z$- and $G$-invariant. By \cref{hehrtgertg}, the Assumption \ref{kopherthretge9}.\ref{kpbgbrgebgrbgrb} implies that $\hat P$   is the spectral projection of the selfadjoint bounded operator
 $$\hat M_{S}:=\frac{1}{|S|}\sum_{g\in G} \hat \rho(g)$$ on $\hat H$ to the isolated  eigenvalue $\{1\}$.

For every measurable subset $Y$ of $\Z\times X$ we let $\hat \nu(Y)$ in $B(\hat H)$ denote the multiplication operator by the characteristic function of $Y$. 
\begin{rem}\label{gkopretgertfert} 
 Let $X_{\disc}$ denote $X$ with the discrete uniform and coarse structures.
The operators $\hat \rho(g)$ and $\hat M_{S}$ are controlled by coarse entourages of  $\Sq(X_{\disc})//G$. Indeed, for every finite subset $F$ of $G$ consider the entourage  $$U_{F}:=\{((n,gx),(n,x))\mid g\in F\ , n\in \Z\ ,x\in X\}$$ of $\Sq(X_{\disc})//G$. Then $\hat \rho(g)$ is $U_{\{g\}}$-controlled, and $\hat M_{S}$ is $U_{S}$-controlled. Since $1$ is an isolated eigenvalue  of $\hat M_{S}$, the operator
$\hat P$ can be approximated by 
controlled operators.  It is  locally finite-dimensional and hence locally compact.  We therefore have $$\hat P\in C_{fp\cap lc}(\Sq(X_{\disc})//G,\hat H, \hat \nu)\ ,$$ where $C_{fp\cap lc}(\dots)$ denotes the $C^{*}$-algebra
of operators that are locally compact and can be approximated by
controlled operators.

The Drutu-Nowak projection $\hat P$ does not directly induce a coarse $K$-homology class in $\pi_{0}K\cX(\Sq(X_{\disc})//G)$ since the Hilbert space $(\hat H,\hat \nu)$ does not admit controlled isometries to ample $ \Sq(X_{\disc})//G$-controlled Hilbert spaces in general. 
 \hB
\end{rem}

The identity map of underlying sets is a morphism $\Sq(X_{\disc})//G\to \Sq(X)//G$ of bornological coarse spaces. In the following 
 we  turn $\hat H$ into a $\Sq(X)//G$-controlled Hilbert space determined on points with locally finite support; see  \cite{buen} for these notions. By Assumption \ref{kopherthretge9}.\ref{hpkprhtrgrtegerg}, we can choose a coarse entourage $V$ of $\Sq(X)$ such that all the components $V_{n}$ are uniform entourages of $X$.  Using the further assumptions on $X$,   we can choose  (see \cref{gkowperweferf} below for details) a measurable  $V^{2}$-bounded partition $(B_{j})_{j\in I}$ of a subset of full measure  of  $\Z\times X$   with $\hat \nu(B_{i})>0$ for all $i$ in $I$ and a locally 
 finite $V^{5}$-dense family $(b_{i})_{i\in I}$ of base points with $b_{i}\in B_{i}$. 
 We define the finitely additive projection-valued measure on $\hat H$ by 
$$\hat\mu(Y):=\sum_{i\in I } \hat \nu(B_{i})\delta_{b_{i}}\ .$$

 \begin{construction}\label{gkowperweferf}{\em 
 We start from a maximal $V$-separated family $(x_{j})_{j\in J}$ on $\Z\times X$. By  Assumption \ref{kopherthretge9}.\ref{hpkprhtrgrtegerg1},
each component $\{n\}\times X$ contains finitely many of them.
 Then $(V^{2}[x_{j}])_{j\in J }$ is a countable covering of $\Z\times X$.
 We can now choose inductively subsets $B_{j}\subseteq V^{2}[x_{j}]$ for every $j$ in $J$
 such that $(B_{j})_{j\in J}$ is a (necessarily $V^{2}$-bounded) measurable partition of $\Z\times X$.
 We set $I:=\{j\in J\mid \nu(B_{j})>0\}$. For every $i$ in $I$, we choose a point $b_{i}$ in the necessarily non-empty set $B_{i}$. Then $(b_{i})_{i\in I}$ is $V^{3}$-dense.
 Indeed, otherwise there exists a point $x$ in $\Z\times X$ such that
 $V[x]\cap B_{i}=\emptyset$ for all $i$ in $I$.
This implies $\nu(V[x])=0$ which contradicts the assumption that $\supp(\nu)=X$. 
 \hB

 }
 \end{construction}

    We  have the Roe algebra  $C(\Sq(X)//G,\hat H,\hat \mu)$ defined as the subalgebra of $B(\hat H)$
 generated by locally finite and controlled operators \cite{buen}.
 The following  observation  was made e.g. in  
  \cite{dsa}, \cite{Li_2023}, \cite{Kitsios:2025aa}.
 Recall our standing Assumption  \ref{kopherthretge9}.\ref{kpbgbrgebgrbgrb}.
 \begin{lem}\label{koprherhertgertgtr} We have\begin{equation}\label{fgdbfgbdfgberth}\hat P\in  C (\Sq(X)//G,\hat H,\hat\mu)\ .
\end{equation}  \end{lem}
\begin{proof} We first consider  the a priori larger Roe type algebra $C_{fp\cap lc}(\Sq(X)//G,\hat H,\hat\mu)$ of locally compact operators that can be approximated by  controlled operators.   The proof of  
  \cite[Thm. 6.20]{mvi} applies to the current situation and ensures the equality 
\begin{equation}\label{werfwerfwerf}C(\hat X//G,\hat H,\hat\mu)=C_{fp\cap lc}(\hat X//G,\hat H,\hat\mu)\ .
\end{equation}
Note that the operators $\hat \rho(g)$ considered on $(\hat H,\hat \mu)$ are no longer $U_{\{g\}}$-controlled (see \cref{gkopretgertfert}), but they are still
$V^{2}\circ U_{\{g\}}\circ V^{2}$-controlled. Consequently, these and $\hat M_{S}$ are controlled on $(\hat H,\hat \mu)$  for the coarse structure of $\Sq(X)//G$. As $\hat P$ is clearly locally compact,  we have  $\hat P\in 
 C_{fp\cap lc}(\Sq(X)//G,\hat H,\hat\mu)$, and therefore by 
 \eqref{werfwerfwerf} we get \eqref{fgdbfgbdfgberth}.
 \end{proof}

\begin{rem}\label{erwg9upwerg}
We use the Assumption \ref{kopherthretge9}.\ref{kpbgbrgebgrbgrb} of a spectral gap for convenience
since then \cref{koprherhertgertgtr} has the classroom-presentable proof given above.
But  the assertion of \cref{koprherhertgertgtr} holds under the more general assumption that the action of
$G$ on $(X,\nu)$ is strongly ergodic  \cite[Def. 2.10]{Li_2021}\footnote{We thank F. Vigolo for pointing this out.}
. In this case it follows from \cite[Cor. 4.4]{Li_2023}
and the equivalence of strong ergodicity and asymptotic expansion  \cite[Lem. 3.16]{Li_2021}. \hB
\end{rem}

  Note that $(\hat H,\hat \mu)$ is determined on points, but is not ample in general. 
  Using  the assumption
that $\hat \nu(B_{i})\not=0$ for all $i$ in $I$  and that the  family $(b_{i})_{i\in I}$ is $V^{5}$-dense, we can conclude that
the  stabilization $$(\hat H^{s},\hat \mu^{s}):=(\hat H\otimes \ell^{2},\hat \mu\otimes \id_{\ell^{2}})$$ is an 
  ample $\widehat X$-controlled Hilbert space in the sense of \cite{buen}. We set
  $\hat P^{s}:=\hat P\otimes e_{0}$ in $C(\Sq(X)//G,\hat H^{s},\hat \mu^{s})$, where $e_{0}$ is the projection onto
  the zeroth basis vector of $\ell^{2}$.

 Recall  that $t: \Sq(X)//G\to \Sq(X)//G$ denotes the shift map.
 We cover $t$ by the  isometry $\hat t:(\hat H^{s},t_{*}\hat \mu^{s})\to (\hat H^{s},\hat \mu^{s})$, where $\hat t:=\hat t(1)$.
 By the $\Z$-invariance of $\hat P$, we have the equality \begin{equation}\label{weferfwerf}\hat t \hat P^{s}\hat t^{*}= \hat P^{s}\ .
\end{equation}

We can now finally construct a coarse  $K$-homology class  $p$  using  \cite[Thm. 6.1]{Bunke:2017aa} (in the case of a trivial group). Since $(\widehat H^{s},\hat \mu^{s})$ is ample on $\Sq(X)//G$, we have a canonical equivalence
$$\kappa_{(\Sq(X)//G,\hat H^{s},\hat \mu^{s})}:K(C(\Sq(X)//G,\hat H^{s},\hat \mu^{s}))\stackrel{\simeq}{\to} K\cX
(\Sq(X)//G)\ .$$  
We let $[\hat P^{s}]$ in $\pi_{0}K(C(\Sq( X)//G,\hat H^{s},\hat \mu^{s}))$ be the class represented by
$\hat P^{s}
$ and define the coarse $K$-homology class
  \begin{equation}\label{whtgwergergws}
 p:=\kappa_{(\Sq(X)//G,\hat H^{s},\hat \mu^{s})}( [\hat P^{s}])
\end{equation} in $\pi_{0}K\cX(\Sq( X)//G)$.
  This finishes the construction of the class $p$. In the follwoing we discuss its essential properties.
    
  Combining the second assertion of  \cite[Thm. 6.1]{Bunke:2017aa} (stating the naturality of the comparison morphisms $\kappa_{\dots}$ in \eqref{whtgwergergws}) and \eqref{weferfwerf} the filler of the square in 
 \cite[(6.1)]{Bunke:2017aa}  determines a witness of 
 the equality  \begin{equation}\label{gwerfwefwerf}t_{*} p =p\ .
\end{equation} 
Since the restriction of $\hat P^{s}$ to every component $\{n\}\times X$ is a one-dimensional projection, we  can determine the trace  and get
$$\tau(p)=(1)_{n}$$ in $\prod_{\Z}\Z$.

 In the following, we use the version of the definition \cite[Def. 8.10]{Bunke:2025aa} of a ghost operator, a notion originally introduced by G.Yu.
We consider $\hat P^{s}$ on $\Sq(X)//G$ with the big family $\Sq_{-}(X)//G$.
\begin{lem} \label{hpopethrtgertgtg} $\hat P^{s}$ is a ghost operator. \end{lem}
\begin{proof} Any coarse entourage of $\Sq(X)//G$ is contained in an entourage  
$U:=W_{F}\circ V$ for some $G$-adapted coarse entourage $V$ of $\Sq(X)$ and $W_{F}$ as in \eqref{oerwigjregvweorgwergwr}.
 We let
$U_{n}$ denote its restriction to the $n$-component which we consider as a copy of $X$.
We must show that
\begin{equation}\label{gerfwefewrfwer}\lim_{n_{0}\to \infty} \sup_{n\ge n_{0}}\sup_{x,x'\in X} \|\nu(U_{n}[x])P\nu(U _{n}[x'])\|=0\ .\end{equation}
We have $$ \|\nu(U_{n}[x])P\nu(U_{n}[x'])\|\le \sqrt{\nu(U_{n}[x])\nu(U_{n}[x'])}\ .$$
 By the $G$-invariance of $\nu$ we furthermore have 
$ \nu(U_{n}[x])\le |F|\sup_{x\in X}\nu(V_{n}[x])$.
We now use that $(V_{n})_{n\in \nat}$ is a uniform scale. 
Since $\nu$ is non-atomic, we can conclude (the argument in \cite[Lem. 6.3]{drno}, \cite[Lem. 5.4]{Li_2023} also works in the context of uniform spaces) that
$\lim_{n_{0}\to \infty}  \sup_{n\ge n_{0}}\sup_{x\in X}\nu(V_{n}[x])=0$. This implies \eqref{gerfwefewrfwer}.
   \end{proof}

\bibliographystyle{alpha}
\bibliography{forschung2021}

\end{document}